\documentclass[12pt,draft,leqno]{article}
\usepackage{amssymb,eucal,latexsym}
\usepackage{amsmath,amsthm,latexsym,amscd,amsbsy,fleqn,graphicx}
\usepackage{xy}
\xyoption{all}

\newtheorem{theorem}{Theorem}[section]
\newtheorem{lm}[theorem]{Lemma}

\newtheorem{cor}[theorem]{Corollary}
\newtheorem{pro}[theorem]{Proposition}
\newtheorem{defi}[theorem]{Definition}

\newtheorem{notas}[theorem]{Notations}
\newtheorem{rem}[theorem]{Remark}

\newtheorem{fact}[theorem]{Fact}
\newtheorem{facts}[theorem]{Facts}

\newtheorem{nist}[theorem]{}
\newtheorem{constr}[theorem]{Construction}

\newcommand{\df}{\ensuremath{\overset{\mathrm{df}}{=}}}

\def\p{\varphi}
\def\a{\alpha}
\def\b{\beta}

\def\ep{\varepsilon}
\def\g{\gamma}

\def\s{\sigma}

\def\ups{\upsilon}
\def\up{\upsilon}

\def\UP{\Upsilon}

\def\epb{\bar{\varepsilon}_A^{Z_A}}

\def\lra{\longrightarrow}

\def\sbe{\subseteq}

\def\stm{\setminus}
\def\ems{\emptyset}
\def\nes{\neq\emptyset}

\def\cuk{\,\check{}\,}

\def\ovl{\overline}

\def\ex{\exists}
\def\fa{\forall}

\def\we{\wedge}

\def\ap{^{\,\prime}}
\def\inv{^{-1}}
\def\st{\ |\ }

\def\nin{\not\in}

\def\cbf{{\bf c}}

\def\1{{\bf 1}}
\def\2{\mbox{{\sf 2}}}
\def\3{\mbox{{\bf 3}}}

\def\AA{{\cal A}}
\def\BB{{\cal B}}

\def\CCC{{\cal C}}

\def\PP{{\cal P}}

\def\XX{{\cal X}}
\def\YY{{\cal Y}}

\def\CCC{{\sf C}}

\def\UUU{{\sf U}}

\def\ult{{\mathfrak u}}
\def\ultv{{\mathfrak v}}
\def\ultw{{\mathfrak w}}
\def\clu{{\mathfrak c}}

\def\CO{{\rm CO}}
\def\KO{{\rm KO}}
\def\RC{{\rm RC}}

\def\CR{{\rm CR}}
\def\BClust{{\rm BClust}}
\def\Clust{{\rm Clust}}
\def\BUlt{{\rm BUlt}}

\def\co{{\sf CO}}
\def\Id{{\sf Id}}
\def\rc{{\sf RC}}
\def\rcl{{\sf RCL}}

\def\bclust{{\sf BClust}}

\def\smf{{\,\smallfrown\,}}
\def\nsmf{{\,\not\smallfrown\,}}

\def\HLC{{\bf LKHaus}}
\def\DHLC{{\bf CLCA}}

\def\CLCA{{\bf CLCA}}

\def\DBoo{{\bf DBoo}}

\def\int{\mbox{{\rm int}}}

\def\cl{\mbox{{\rm cl}}}

\def\Clust{\mbox{{\rm Clust}}}

\def\BClu{\mbox{{\rm BClust}}}
\def\Ult{\mbox{{\rm Ult}}}
\def\Ultsf{\mbox{{\sf Ult}}}

\def\COsf{\mbox{{\sf CO}}}

\def\doc{\hspace{-1cm}{\em Proof.}~~}
\def\sq{\hspace*{\fill} \hbox{\vrule\vbox{\hrule\phantom{o}\hrule}\vrule}}
\def\sqs{\sq \vspace{2mm}}

\def\Top{{\bf Top}}

\def\BBBB{\mathbb{B}}

\def\Set{{\bf Set}}

\def\tcx{t_X^C}
\def\tcy{t_Y^C}
\def\tcx0{t_{(X,X_0)}}
\def\tcy0{t_{(Y,Y_0)}}

\def\di{\diamond}

\def\bU0{\bar{U}=(U^0,(U^i,U^{ci})_{i\in\omega})}

\def\bV0{\bar{V}=(V^0,(V^i,V^{ci})_{i\in\omega})}

\def\CAPX{{\sf C}(\AA,\PP,\XX)}

\def\ZCB{{\bf zCBoo}}
\def\LZCB{{\bf lzCBoo}}

\def\EDT{{\bf EdTych}}
\def\EDL{{\bf EdLKH}}

\title{{\LARGE\bf Categorical Extension of Dualities: }\\
\vspace{0.2cm}
{\LARGE\bf From Stone to de Vries and Beyond, II}\\
\vspace{0.5cm}
{\large\bf G. Dimov, E. Ivanova-Dimova
and W. Tholen}\thanks{The  first  two authors acknowledge the   support
by Bulgarian National Science Fund, contract no. DN02/15/19.12.2016.
The third author acknowledges the support under Discovery Grant no. 501260 of the Natural Sciences and Engineering Council of Canada.}
\\
\vspace{0.2cm}
{\footnotesize\rm Faculty of Mathematics and Inf.,  Sofia University,}
{\footnotesize\rm 5 J. Bourchier Blvd., 1164 Sofia, Bulgaria}\\
{\footnotesize\rm Dept. of Mathematics and Statistics, York University,}
 {\footnotesize\rm Toronto, Ontario, M3J 1P3, Canada}
\\
\vspace{0.5cm}
{\normalsize\em Dedicated to the memory of Professor Mitrofan Choban}}

\author{}

\date{}

\begin{document}

\maketitle

\begin{abstract}
Under a general categorical procedure for the extension of dual equivalences as presented in this paper's predecessor, a new algebraically defined category is established that is dually equivalent to the category $\bf LKHaus$ of locally compact Hausdorff spaces and continuous maps, with the dual equivalence extending a Stone-type duality for the category of extremally disconnected locally compact Hausdorff spaces and continuous maps. The new category is then shown to be isomorphic to the category $\bf CLCA$ of complete local contact algebras and suitable morphisms. Thereby, a new proof is presented  for the equivalence ${\bf LKHaus}\simeq{\bf CLCA}^{\rm op}$  that was obtained by the first author more than a decade ago. Unlike the morphisms of $\bf CLCA$, the morphisms of the new category and their composition law are very natural and easy to handle.
\end{abstract}



\footnotetext[3]{{\footnotesize {\em E-mail addresses:}
gdimov@fmi.uni-sofia.bg, elza@fmi.uni-sofia.bg, tholen@yorku.ca}.}

\section{Introduction}
\label{introl}

The de Vries duality \cite{deV} extends the restricted Stone duality between extremally disconnected compact Hausdorff spaces and complete Boolean algebras to all compact Hausdorff spaces and so-called de Vries algebras; these are complete Boolean algebras equipped with a structure, here taken in its equivalent form of a normal contact relation as defined in \cite{F,DV1}. They are the objects of the category $\mathbf{deV}$ whose morphisms are somewhat unusual and cumbersome to handle, since they may not respect the Boolean structure, and their categorical composition generally does not proceed by map composition.

In this paper's predecessor \cite{DDT1} we presented the de Vries dual equivalence
\begin{center}
$\xymatrix{{\bf deV}^{\rm op}\ar@/^0.6pc/[rr] & {\scriptstyle \simeq} & {\bf KHaus}\ar@/^0.6pc/[ll].
}$
\end{center}
as the composite of an isomorphism and two equivalences, as in
\begin{center}
$\xymatrix{{\bf deV}^{\rm op}\ar@/^0.2pc/[r] & ({\bf deVBoo}/\!\backsim)^{\rm op}\ar@/^0.2pc/[r]\ar@/^0.2pc/[l] & ({\sf C}({\cal A},{\cal P},{\cal X})/\!\sim)^{\rm op}\ar@/^0.2pc/[r]\ar@/^0.2pc/[l] & {\bf KHaus\;,}\ar@/^0.2pc/[l]\\
}$	
\end{center}
where
\begin{itemize}
	\item the (comma) category ${\sf C}({\cal A},{\cal P},{\cal X})$ was obtained categorically from the above-mentioned restricted Stone duality between the categories $\cal X=\mathbf{EKHaus}$
 and $\cal A=\mathbf{CBoo}$,
 augmented by the class $\cal P$ of irreducible (\cite{Wh,Gle,ArP}) continuous maps of compact Hausdorff spaces with extremally disconnected domain;
	\item the equivalence relation $\sim$ on ${\mathsf C}(\mathcal A,\mathcal P,\mathcal X)$ arose very naturally, in such a way that the quotient category became (dually) equivalent to $\mathbf{KHaus}$, as a formal categorical extension of the equivalence ${\mathcal A}^{\rm op}\simeq \mathcal X$;
	\item the category ${\sf C}({\cal A},{\cal P},{\cal X})$ could be shown to be equivalent to a new category
named $\bf deVBoo$, which has the same objects as the category $\bf deV$, but whose morphisms are Boolean morphisms reflecting the contact relation, and being categorically composed by ordinary map composition;
	\item the equivalence relation $\backsim$ was just the ``translation'' of $\sim$ along this last equivalence.	
	\end{itemize}
In this way we not only obtained an alternative proof of the original de Vries duality theorem but, largely based on our general categorical construction, also established a new dual equivalence of $\bf KHaus$ with the  category ${\bf deVBoo}/\!\backsim$ whose morphisms and their compositions struck us as described more naturally than those of the original category $\bf deV$.

In this paper we aim to establish analogous results for the category $\bf LKHaus$ of locally compact Hausdorff spaces and continuous maps which, in \cite{D-AMH1-10}, the first author showed to be dually equivalent to the category $\bf CLCA$ of complete local contact algebras (\cite{Roeper,VDDB,DV1}) and suitable morphisms; the objects of the category  $\bf CLCA$  are generalized de Vries algebras equipped with an ideal of their underlying Boolean algebra which satisfies some natural conditions, and their morphisms are generalized de Vries morphisms satisfying some compatibility conditions with the ideal structure.
Our main goal is the construction of a new category, dually equivalent to the category $\bf LKHaus$,  with the same objects as $\bf CLCA$, but with more naturally described morphisms, composed in a standard manner.\footnote{We note that in \cite{DI} another category dually equivalent to $\bf LKHaus$
 is presented. While its composition law may be considered to be more natural than that of the category $\bf CLCA$, its morphisms, which are special multi-valued maps, may not.}

  To this end we follow a path similar to the procedure used in our alternative proof of the de Vries duality, by first formally extending a  Stone-type dual equivalence which follows immediately from the results in \cite{DD,DD1}. Specifically, taking now for $\cal X$ the category $\bf EdLKH$ of extremally disconnected locally compact Hausdorff spaces, we  describe its Stone-type dual  as the category $\cal A= \bf lzCBoo$; its object are  complete lz-algebras, these being complete Boolean algebras equipped with an open dense subset of its Stone dual (see \cite{DD,DD1}).
Choosing for $\cal P$ the  class of   perfect irreducible maps of locally compact Hausdoff spaces with extremally disconnected locally compact Hausdorff domain, we are now at the beginning of a passage that culminates again in the establishment of a string of an isomorphism and two equivalences, indicated by
\begin{center}
$\xymatrix{{\bf CLCA}^{\rm op}\ar@/^0.2pc/[r] & ({\bf DBoo}/\!\backsim)^{\rm op}\ar@/^0.2pc/[r]\ar@/^0.2pc/[l] & ({\sf C}({\cal A},{\cal P},{\cal X})/\!\sim)^{\rm op}\ar@/^0.2pc/[r]\ar@/^0.2pc/[l] & {\bf LKHaus\;,}\ar@/^0.2pc/[l]\\
}$
\end{center}
whose composite is the dual equivalence of \cite{D-AMH1-10}.  Here, $\bf DBoo$ is a category  with the same objects as  the category $\bf CLCA$ and with morphisms that are Boolean homomorphisms  reflecting the contact relation and respecting the ideal structure,  composed by ordinary map composition. Its quotient category gives a new dual equivalence with $\bf LKHaus$ (see Theorem \ref{md}), and by showing that it is isomorphic to $\bf CLCA$, we finally obtain an alternative proof of the first author's duality result \cite{D-AMH1-10} (see Theorem \ref{newlcdth} and Corollary \ref{cornlcdth}).

This Introduction is followed by a section in which we collect all notational ingredients and known preliminary facts needed for the establishment of the above string of equivalences, referring to the literature for proofs of these facts. Section 3 is devoted to proving the new dual equivalence of $\bf LKHaus$ with the category ${\bf DBoo}/\!\backsim$ (see Theorem \ref{md}). Here, among other assertions, we show (see Proposition \ref{eqrel}) that if
 $A$ is a local contact algebra with contact relation $\smf$ and ideal $\BBBB$ of ``bounded'' elements,
then the extension of
 the relation $\smf$ to ultrafilters in $A$
 is an equivalence relation on the set\/ $\BUlt(A)$ of all bounded ultrafilters in $A$ (i.e., of those ultrafilters $\ult$ for which $\ult\cap\BBBB\nes$). This fact generalizes  a result for normal contact algebras $A$ established in \cite{DV1,DUV}, asserting that the extension of the contact relation to ultrafilters
 is  an equivalence relation on $\Ult(A)$.
As another important step towards establishing our duality results, regarding
the set $Z=\BUlt(A)$ as a subspace of the Stone dual $\Ultsf(A)$ of $A$, in Proposition \ref{eqrel} we also prove that
$Z$ is an open dense subset of the space $\Ultsf(A)$. In particular, $Z$ is a locally compact Hausdorff space, and if $A$ is a complete local contact algebra then $Z$ is extremally disconnected,
and the natural quotient map $p_{A}:Z\to Z/\smf$   is perfect irreducible, with a locally compact Hausdorff quotient space.

Section 4 completes the establishment of the above string of equivalences and the new proof of the duality of \cite{D-AMH1-10} (see Theorem \ref{newlcdth} and Corollary \ref{cornlcdth}). In particular, in the above notation, we show that the space $Z/\smf$ is homeomorphic to the space that appeared as the dual of the complete local contact algebra $A$ in the duality  of \cite{D-AMH1-10}, that is: the space of bounded clusters in
$A$ (see Proposition \ref{full propl}). An important ingredient to the proof of Theorem \ref{newlcdth} is the following crucial result: for every $\CLCA$-morphism $\a:(A,\smf,\BBBB)\lra (A\ap,\smf',\BBBB\ap)$ there exists a $\bf DBoo$-morphism $\p: (A,\smf,\BBBB)\lra (A\ap,\smf',\BBBB\ap)$ (i.e., a Boolean homomorphism $\p:A\to B$
which reflects the contact relation, and  for every $b\ap\in\BBBB\ap$ there exists $b\in\BBBB$ with $b\ap\le\p(b)$) such that
$$\a(a)=\bigvee\{\p(b)\st b\in\BBBB, b\nsmf a^*\}$$
 for every $a\in A$; here $a^*$ denotes the complement of $a$. In Proposition \ref{V} we also show that, conversely, if $\p$ is a $\DBoo$-morphism, then the function $\a$, defined by the above formula, is a $\CLCA$-morphism.

We note that, with the help of a simplified variant (obtained in \cite{DDT}) of our general categorical procedure for the extension of dual equivalences from \cite{DDT1}, we gave in \cite{DDT} a new proof of the Fedorchuk Duality Theorem \cite{F}.
Other applications of the general categorical procedure for the extension of dual equivalences will be presented in \cite{DDT3}.

If ${\mathcal C}$ is a category, then we denote by $|{\mathcal C}|$ the class of the objects of ${\mathcal C}$ and by ${\mathcal C}(X,Y)$ the set of all
  ${\mathcal C}$-morphisms between two ${\mathcal C}$-objects $X$ and $Y$.

Our general references for unexplained notation and notions are \cite{AHS} for category theory,
 \cite{E} for topology, and \cite{kop89} for Boolean algebras.

\section{Preliminaries}\label{s2nn}
For the reader's convenience, we recall from \cite{D-AMH1-10} some notational conventions and background facts and extend them as needed for this paper.

\begin{facts}\label{Stonel}
\rm
{\em (The Stone duality)}.
The functor ${\sf CO}:{\bf Top}\longrightarrow {\bf Boo}^{{\rm op}}$
 assigns to every topological space $X$
 its Boolean algebra $(\mbox{{\rm CO}}(X),\subseteq )$ of closed and open subsets and operates on morphisms by taking inverse images of sets.
 Its restriction to Stone spaces (= zero-dimensional compact Hausdorff spaces) is part of
 the Stone duality \cite{ST},
 $${\sf Ult} :{\bf Boo}^{{\rm op}}\longrightarrow {\bf Stone} \ \ \mbox{ and } \ \ {\sf CO} :{\bf Stone}\longrightarrow  {\bf Boo}^{{\rm op}}.$$
The other part, $\sf Ult$,
assigns to a Boolean algebra $A$ the
space ${\sf Ult} (A)$
of ultrafilters in $A$ and operates on morphisms by taking inverse images of ultrafilters. The
topology of ${\sf Ult}(A)$
takes
the family $\{\ep_A(a)\ |\  a\in A\}$ as a base for closed sets, where
$$\ep_A(a)\ensuremath{\overset{\mathrm{df}}{=}}
\{\ult\in \Ult(A)\ |\  a\in\ult\},$$
for every $a\in A$ . (Note that the  family $\{\ep_A(a)\ |\  a\in A\}$ is also an open base for
the topology of $\Ultsf(A)$.)
The {\em Stone map}
$\ep_A:A\longrightarrow  {\sf CO} ({\sf Ult} (A))$
is a ${\bf Boo}$-isomorphism and, in our setting, serves as
 the counit of the Stone dual equivalence at
 $A$.
The adjunction unit
at a Stone space $X$ is the
homeomorphism
$$\eta_X: X\longrightarrow{\sf Ult}({\sf CO}(X)),\quad x\longmapsto \ult_x\df\{U\in{\sf CO}(X)\,|\,x\in U\}.$$

It is well-known
(\cite{Halmos,kop89})
that, under the Stone duality, {\em complete} Boolean algebras correspond  to compact Hausdorff spaces that are {\em extremally disconnected} (so that the closure of any open set is still open). Hence, in a self-explanatory notation, one has the dual equivalence
\begin{center}
$\xymatrix{{\bf CBoo}^{\rm op}\ar@/^0.6pc/[rr]^{\sf Ult\;\;\;} & {\scriptstyle \simeq} & {\bf EdKH}\ar@/^0.6pc/[ll]^{\sf CO\;\;\;}.
}$
\end{center}
\end{facts}

\begin{facts}\label{deVriesAlgebraicl}
\rm
A pair $(A,\smf)$, where $A$ is a Boolean algebra and $\smf$ is a relation on it, is called a {\em contact algebra}\/ (\cite{DV1}) if it satisfies the following conditions:
\begin{enumerate}
\renewcommand{\theenumi}{\ensuremath{{\rm C}\arabic{enumi}}}
 \item $a\smallfrown a$ whenever $a>0$;\label{c1}
    \item $a\smallfrown b$ implies $a>0$ and $b>0$;\label{c2}
     \item $a\smallfrown b$ implies $b\smallfrown a$;\label{c3}
     \item $a\smallfrown(b\vee c)$  if, and only if, \!  $a\smallfrown b$ or $a\smallfrown c$. \label{c4}
\end{enumerate}
With the {\em non-tangential inclusion}\/ relation $\ll$ (denoted sometimes by $\ll_\smf$ as well) on $A$ defined by
$(a\ll b\iff a\not\smallfrown b^*),$
these conditions may equivalently be stated as
\begin{enumerate}
\renewcommand{\theenumi}{{\rm I\,}\arabic{enumi}}
\item $a \ll b \text{ implies } a \leq b$; \label{di1}
\item  $a\leq b\ll c\leq d$ implies $a\ll d$; \label{di3}
\item $0\ll 0$, and $a\ll b$ implies $b^*\ll a^*$; \label{di7}
\item $a\ll c$ and $b\ll c$ implies $a\vee b\ll c$. \label{di4}
\end{enumerate}

 In terms of $\ll$, the contact relation $\smallfrown$ takes the form $(a\smallfrown b\iff a\not\ll b^*)$.

A contact algebra $A$ is said to be a {\em normal contact algebra}\/ (see \cite{deV,F}, where, however, different names have been used) if it also satisfies  the condition

\smallskip

\noindent \ \  I5. if  $a\ll c$ and $c>0$, then $a\ll b\ll c$  for some $b>0$.

\smallskip

The typical examples of  (normal) contact algebras can be obtained as follows.

A set $F$ in a topological space $X$ is {\em regular closed} (or a {\em closed domain} \cite{E}) if it is the closure of its interior in $X$: $F={\rm cl(int}(F))$. The collection ${\RC}(X)$ of all regular closed sets in $X$ becomes a Boolean algebra, with the Boolean operations $\vee,\wedge,\,^*,0,1$ given by
\begin{align*}
F\vee G &= F\cup G, & F\wedge G &= \mbox{{\rm cl}}(\mbox{{\rm int}}(F\cap G)), & F^* &=  \mbox{{\rm cl}}(X\setminus F), & 0 &=  \emptyset, & 1 &=  X.
\end{align*}
The Boolean algebra ${\sf RC}(X)$ is actually complete, with the infinite joins and meets given by
\begin{align*}
 \bigvee_{i\in I}F_i &=  \mbox{{\rm cl}}(\bigcup_{i\in I}F_i)\
=\mbox{{\rm cl}}(\bigcup_{i\in I}\mbox{{\rm int}}(F_i))=\mbox{{\rm cl}}(\mbox{{\rm int}}(\bigcup_{i\in I}F_i))\quad \text{and}\quad
\bigwedge_{i\in I}F_i &= \mbox{{\rm cl}}(\mbox{{\rm int}}(\bigcap_{i\in I}F_i)).
 \end{align*}
 With the {\em contact relation}
 $\smallfrown_X$ given by
$$F\smallfrown_X G\iff F\cap G\neq\emptyset,$$
$({\sf RC}(X),\smf_X)$ becomes a  contact algebra and if $X$ is a normal Hausdorff space, then $({\sf RC}(X),\smf_X)$ becomes even a normal contact algebra (see \cite{deV,DV1}).

In $({\sf RC}(X),\smallfrown_X)$, the associated relation $\ll_X$ reads as
$$F\ll_X G\iff F\subseteq \mbox{{\rm int}}_X(G).$$
\end{facts}

\begin{facts}\label{deVriesl}
\rm
Extending the contact relation
$\smallfrown$ on a Boolean algebra $A$ to a relation on the set $\Ult(A)$ of
 its ultrafilters
 (which, for brevity, will again be denoted by $\smallfrown$)
 by
$$\mathfrak u\smallfrown\mathfrak v\iff \forall\; c\in\mathfrak u,\; d\in\mathfrak v:c\smallfrown d,$$
one shows (\cite[Lemma 3.5, p. 222]{DV1})
that the contact relation for elements $a,b\in A$ is characterized by its ultrafilter extension, via
$$a\smallfrown b\iff \exists\; \mathfrak{u, v}\in{\rm Ult}(A): a\in\mathfrak u,\; b\in\mathfrak v,\; \mathfrak u\smallfrown\mathfrak v.$$
Furthermore, {\em if $A$ is normal, then \  $\smallfrown$ \ is an equivalence relation on\/ ${\rm Ult}(A)$} (\cite{DV1,DUV}).
 We will often use the concept of {\em cluster}\/ (\cite{LE,DV1})  in a normal contact algebra $A$; this is a subset $\mathfrak c$ of $A$ satisfying the following conditions for all $a,b\in A$:
\begin{enumerate}
\renewcommand{\theenumi}{cl\,\arabic{enumi}}
\item $\mathfrak c\neq\emptyset$;
\item\label{cl1}  $a,b\in\mathfrak c$ implies $a\smallfrown b$;
\item\label{cl2} $a\vee b\in\mathfrak c$ implies $a\in\mathfrak c$ or
$b\in\mathfrak c$;
\item\label{cl3} if $a\smallfrown b$ for all $b\in\mathfrak c$, then $a\in\mathfrak c$.
\end{enumerate}

\noindent We denote the set of all clusters in a normal contact algebra $A$ by ${\rm Clust}(A)$.
As an easy consequence of the definition of a cluster one has the property
$(a\in\mathfrak c,\, a\le b\Longrightarrow b\in\mathfrak c)$ (i.e., they are upwards closed).
Proceeding as in the proof of Theorem 5.8 of \cite{NW} one shows that {\em every ultrafilter\/ $\mathfrak u$ in a normal contact algebra $A$ generates the cluster}
$${\mathfrak c}_{\mathfrak u}=\{a\in A\,|\,\forall\, b\in\mathfrak u: a\smallfrown b\},$$
and {\em every cluster $\mathfrak c$ in $A$ comes about this way}, that is: $\mathfrak c={\mathfrak c}_{\mathfrak u}$, for some $\mathfrak u\in{\rm Ult}(A)$; actually, {\em for  every $a\in \mathfrak c$ one has $\mathfrak c={\mathfrak c}_{\mathfrak u}$, for some $\mathfrak u\in{\rm Ult}(A)$ with $a\in\mathfrak u$.} One concludes that $\mathfrak c_{\mathfrak u}$ {\em is the unique cluster containing a given ultrafilter} $\mathfrak u$, and that {\em any two clusters comparable by inclusion must actually be equal.} Most importantly, the
relation \ $\smallfrown$ \ for ultrafilters is characterized by $$\mathfrak u\smallfrown\mathfrak v\iff \mathfrak c_{\mathfrak u}=\mathfrak c_{\mathfrak v}.$$

We will also use the concept of {\em clan}\/ in a contact algebra $A$: clans are those  subsets $\mathfrak c$ of $A$ which satisfy conditions (cl1)-(cl3) and are upward closed (see \cite{Thron,DV1}). Clusters and clans permit the following characterization of contact relation: {\em if $(A,\smf)$ is a contact algebra (respectively, normal contact algebra) then, for every $a,b\in A$, $a\smf b$ if, and only if, there exists a clan (respectively, cluster) in $(A,\smf)$ which contains $a$ and $b$} (\cite[Proposition 3.3 and Corollary 3.4, p.222]{DV1}). Also, {\em every clan coincides with the union of all ultrafilters contained in it} (\cite[Fact 3.3(v), p.219]{DV1}).
\end{facts}

In \cite{Roeper}, Roeper introduced the notion of {\em region-based
topology}, as a Boolean
algebra provided with a contact relation
and an one-place
predicate of {\em boundedness}.  His  Representation Theorem (see Theorem \ref{roeperl} below)
implies the existence of a bijective correspondence between the class of all (up
to isomorphism) region-based topologies and the class of all (up to
homeomorphism) locally compact Hausdorff spaces.
The axioms of a region-based topology almost coincide with those
  of a {\em local proximity space}\/ as introduced by Leader
\cite{LE}.  This similarity led us to calling
region-based topologies {\em local
contact algebras}\/ (\cite{VDDB, DV1}) since, with the methods of local proximity spaces, it enabled us to give
 a shorter
proof of Roeper's  Representation
Theorem (\cite{VDDB, D-diss}).
Below we recall the needed definitions and assertions concerning local contact algebras.

\begin{nist}\label{locono}
\rm
A {\it  local contact algebra} is a contact algebra
$(A,\smallfrown)$, provided with a (not necessarily proper) ideal $\BBBB$ of $A$ whose elements we call {\em bounded}, satisfying
the following axioms (where $\ll$ is as in \ref{deVriesAlgebraicl}):

\medskip

\noindent(BC1) if $a\ll c$ with $a$ bounded, then $a\ll b\ll c$ for some bounded element $b$\,;\\
(BC2) if $a\smallfrown b$, then $a\smallfrown (c\we b)$ for some bounded element $c$\,;\\
(BC3) for every element $a\neq 0$ there is a bounded element  $b\neq 0$ with $b\ll a$.

\medskip
\noindent The local contact algebra is {\em complete}\/ if its underlying Boolean algebra is so. It is easy to see that {\em the local contact algebras in which all elements are bounded are equivalently described as the normal contact algebras of}\/ Facts \ref{deVriesAlgebraicl}.

In the sequel, for brevity, writing $``A$ (respectively, $A_1$, etc.) is a local contact algebra'', we will mean that the notation for all its components is fixed to be $(A,\smf,\BBBB)$ (respectively, $(A_1,\smf_1,\BBBB_1)$, etc.). Of course, there are many local contact algebras with one and the same underlying Boolean algebra, but our notation should not cause problems.
\end{nist}

From \cite{VDDB} we recall
the lattice-theoretical
counterparts of some theorems of Leader's paper \cite{LE}.

\begin{defi}\label{Alexprn}{\rm (\cite{VDDB})}
For a local contact algebra $A$
one defines the {\em Alexandroff extension} $\smallfrown_{\rm\! Al}$ of\/
$\smallfrown$ by
$$ a \smf_{\rm\! Al} b\;\Longleftrightarrow\; (a\smallfrown b\ \mbox{ or }\ a,b\in A\stm\BBBB).$$
\end{defi}

\begin{lm}\label{Alexprn1}{\rm (\cite{VDDB})}
The Alexandroff extension $\smallfrown_{\rm\! Al}$ of the contact relation $\smallfrown$ of a local contact algebra $A$ makes $(A,\smallfrown_{\rm\! Al})$
a normal contact algebra.
\end{lm}

\begin{defi}\label{boundcl}
\rm
By a {\em cluster in} a local contact algebra $A$
we mean a cluster in the normal contact algebra
$(A,\smf_{\rm\! Al})$. We call a cluster $\mathfrak c$ in $A$
{\em bounded}\/ if
$\mathfrak c\cap\BBBB\nes$, and {\em unbounded} otherwise; likewise for an ultrafilter $\mathfrak u$ in the underlying Boolean algebra of $A$.
The set of all
bounded clusters (of all bounded ultrafilters) in $A$
is denoted by
$\BClu(A)$ (by $\BUlt(A)$, respectively).
\end{defi}

\begin{lm}\label{neogrn}{\rm (\cite{VDDB})}
Let $A$ be a local contact algebra  with
$1\not\in\BBBB$. Then the set ${\mathfrak c}_\infty\df A\stm \BBBB$ is a cluster in $A$; in fact, it is
 the only unbounded cluster in $A$: $$\BClu(A)=\Clust(A,\smf_{\rm\! Al}) \stm \{{\mathfrak c}_\infty\}.$$
\end{lm}

\begin{pro}\label{stanlocn}{\rm (\cite{Roeper,VDDB})}
Let $X$ be a locally compact Hausdorff space. Then:

\smallskip

\noindent{\rm(a)} The  contact algebra
$(\RC(X),\smallfrown_{X})$ becomes a complete local contact algebra with bounded elements given by the ideal  $\CR(X)$ of all compact regular closed sets in $X$; we  call it
the\/ {\em standard local contact algebra of} $X$.

\smallskip

\noindent{\rm (b)} For every $x\in X$, the set $\{F\in \RC(X)\st x\in F\}$ is a bounded cluster in the
standard local contact algebra of $X$.
\end{pro}

Recall that, in a poset $(A,\le)$
with bottom element $0$, a subset $M\sbe A$ is said to be {\em dense}\/ in $A$ if for every $0\neq a\in A$ there is an $0\neq m\in M$ such that $m\le a$; a map $f:A\to B$
of posets
is said to be {\em dense}\/ if $f(A)$ is dense in $B$.

\begin{theorem}\label{roeperl}
{\rm (Roeper's Representation Theorem for local contact algebras (\cite{Roeper}))}

\noindent{\rm (a}) Every local contact algebra
may be densely embedded
into the
 standard local contact algebra
 of a locally compact Hausdorff
 space by a Boolean monomorphism preserving and reflecting the contact relation and boundedness.
 Moreover, when the algebra is complete, the embedding becomes a Boolean isomorphism.

\smallskip

\noindent{\rm (b)}
There exists a bijective correspondence
 between the class of all complete local contact algebras (up to Boolean isomorphism preserving and reflecting the contact relation and boundedness)
 and the class of all locally compact Hausdorff spaces (up to homeomorphism).
\end{theorem}

\noindent{\em Proof (Sketch).} (a) For a locally compact Hausdorff space $X$, we let
$ \rcl(X)$
denote its standard local contact algebra.
For a local contact algebra $A$, we
denote by $\bclust(A)$ the topological space
of all bounded clusters in $A$ with a topology for which  the family
$$\{\tau_{A}(a)\df \{{\mathfrak c}\in\BClust(A)\st a\in {\mathfrak c}\}\st a\in A\}$$
is a base for closed sets.
Then:

\smallskip

\noindent(i)  The space $L=\bclust(A)$ is locally compact Hausdorff. Denote by $X$ the set $\Clust(A,\smf_{\rm\! Al})$ endowed with a topology having the family
$\{\{{\mathfrak c}\in\Clust(A,\smf_{\rm\! Al})\st a\in {\mathfrak c}\}\st a\in A\}$
as a base for closed sets. Then $1\nin\BBBB$ implies that $X$ is the Alexandroff (one-point) compactification of $L$; if $1\in\BBBB$, then
 $L\equiv X$.

\smallskip

\noindent(ii) $\tau_{A}$ is a dense Boolean monomorphism of the Boolean
algebra $A$ into
the Boolean algebra $\RC(L)$, and if $A$ is complete, then $\tau_{A}$
is a Boolean isomorphism onto the Boolean algebra $\RC(L)$;

\smallskip

\noindent(iii) an element $a\in A$ is bounded if, and only if, $\tau_A(a)$ is compact: ($a\in\BBBB \iff \tau_{A}(b)\in \CR(L))$;

\smallskip

\noindent(iv)
$(\;a\smallfrown b \iff \tau_{A}(a)\cap \tau_{A}(b)\neq\ems\;)$ for all $a,b\in A$.

\smallskip

\noindent(v)  In other words, $\tau_{A}$ embeds the local contact
algebra $A$ into the standard local contact algebra
$\RC(L)$, with the embedding preserving and reflecting the contact relation and boundedness; moreover,
if $A$ is complete, then $\tau_A: A\lra \RC(L)$
is a Boolean isomorphism.

\medskip

\noindent(b)
Let $L$ be a locally compact Hausdorff space. Then the map
$$ \s_L:L\lra\bclust(\rcl(L)),\;x\longmapsto\{F\in \RC(L)\st x\in
F\},$$  is a homeomorphism.
This fact and item (v) above
imply our assertion.
\sqs

Next we recall the extension of de Vries' Duality Theorem \cite{deV} as obtained in \cite{D-AMH1-10}.

\begin{defi}\label{dhc}{\rm \cite{D-AMH1-10}}
\rm
The objects of the category $\DHLC$ are all complete local contact algebras; its
morphisms are maps $\p:A_1\lra
A_2$
satisfying the following
conditions:

\smallskip

\noindent(CLC1) $\p(0)=0$;\\
(CLC2) $\p(a\we b)=\p(a)\we \p(b)$, for all $a,b\in A_1$;\\
(CLC3) if
$a\ll b$ in $A_1$ with $a$ bounded, then $(\p(a^*))^*\ll
\p(b)$;\\
(CLC4) for every bounded element $b\in A_2$ there is a bounded $a\in A_1$ with
$b\le\p(a)$;\\
\noindent(CLC5) $\p(a)=\bigvee\{\p(b)\st b\in A_1 \text{ bounded, } b\ll a\}$,
for every $a\in A_1$.

\medskip

{\noindent}The composite $\psi\diamond\p$ in $\DHLC$ of morphisms
$\p:A_1\lra A_2$ and
$\psi:A_2\lra A_3$
is given by
$$ \psi\diamond\p = (\psi\circ\p)\cuk,$$
 where, for any
function $f:A\to A\ap$ of
$\DHLC$-objects, we have put $f\cuk:A\to
A\ap$ with
$$ f\cuk(a)=\bigvee\{f(b)\st b\in A\text{ bounded, }\; b\ll a\},$$
for every $a\in A$.
\end{defi}

Here are some immediate consequences
  of the conditions (CLC1)-(CLC5):

 \begin{lm}\label{pf1}{\rm \cite{D-AMH1-10}}
Let $\p:A\to A\ap$ be a $\CLCA$-morphism.  Then:

\smallskip

\noindent{\rm (a)} for every $a\in A$,
$\p(a^*)\le (\p(a))^*$ and $\p\cuk(a)\le\p(a)$;\\
{\rm (b)} if
$b_1\ll a_1$ and $b_2\ll a_2$ with $b_1,b_2$ bounded, then $\p(b_1\vee b_2)\ll \p(a_1)\vee \p(a_2)$.
\end{lm}

We will also need the following simple technical lemma from \cite{D-AMH1-10}:

\begin{lm}\label{assocuk}
For functions  $f:A_1\to A_2$, $g: A_2\to A_3$ of\/ $\bf CLCA$-objects one has:

\smallskip

\noindent{\rm (a)} $(g\cuk\circ f)\cuk=(g\circ f)\cuk$;

\noindent{\rm (b)} If $f$ and $g$ are monotone, then
$(g\circ f\cuk)\cuk=(g\circ f)\cuk$.
\end{lm}

In the proof of the next theorem we will use the following standard notation:  $\uparrow\! M\df\{a\in A\st m\le a\ \mbox{ for some }\ m\in M\}$, where  $(A,\le)$ is a poset and $M\sbe A$.

\begin{theorem}\label{lccont}{\rm (The Duality Theorem of  {\rm \cite{D-AMH1-10}} for the category $\HLC$)}
$\rcl$ and $\bclust$ (of the proof of Theorem \ref{roeperl}) are the object assignments of contravariant functors that make
the categories $\HLC$ and\/ $\DHLC$ dually equi\-valent.
\end{theorem}

\noindent{\it Proof (Sketch).} Defining $\rcl$ on a morphism $f:X\to Y$ in $\HLC$, one puts
$$ \rcl(f)(G)=\cl(f\inv(\int(G))),$$
for every
$G\in \RC(Y)$, and for $\p:A\to A'$ in $\CLCA(A,A\ap)$ one defines $\bclust(\p)$ by
$$ \bclust(\p)({\mathfrak c}\ap)=\, \uparrow\!\{a\in \BBBB\st \mbox{for all } b\in A
\mbox{ with } a\ll b:\; \p(b)\in{\mathfrak c}\ap\}$$
for
all ${\mathfrak c}\ap\in\bclust(A')$.
The natural isomorphisms
$$\s:{\sf Id}_{\,\HLC}\lra\bclust\circ\rcl \mbox{ and  }\tau: {\sf Id}_{\,\DHLC}\lra\rcl\circ\bclust$$
(as defined in the proof of Theorem \ref{roeperl}) constitute respectively the unit and counit of the dual adjunction $\rcl\dashv\bclust:{\bf CLCA}
\lra{\bf LKHaus}$.
\sqs

Next we recall some useful assertions about (bounded) clusters in local contact algebras.
Here, for an ultrafilter
 ${\mathfrak u}$ in a local contact algebra $A$,
 we denote by\/ ${\mathfrak c}_{\mathfrak u}$ the cluster in the normal contact algebra $(A,\smf_{\!\rm Al})$ generated by\/ $\ult$ (see {\rm Facts \ref{deVriesl}} and {\rm Lemma \ref{Alexprn1}}).

\begin{pro}\label{bbclf}{\rm \cite{D-AMH1-10}}
 If, for an ultrafilter \/ ${\mathfrak u}$ in a local contact algebra $A$,
the cluster ${\mathfrak c}_{\mathfrak u}$
contains a bounded element, then so does $\mathfrak u$ itself.
\end{pro}

\begin{pro}\label{bbcl1}{\rm \cite{D-AMH1-10}}
Let\/ ${\mathfrak c}$ be a bounded cluster in a local contact algebra $A$. Then:

\smallskip

\noindent{\rm (a)} For every $a\in{\mathfrak c}$ there exists a bounded element $c\in{\mathfrak c}$ with $c\le a$;

\noindent{\rm (b)} an element $a\in A$ lies outside $\mathfrak{c}$ iff there is a bounded element $b\in \mathfrak{c}$
with  $b\ll a^*$.
\end{pro}

\begin{facts}\label{projectivel}
\rm
Recall  that if $X$ is a regular Hausdorff space then a space $EX$ is called an
{\em absolute of}\/ $X$ if there exists a perfect irreducible map $\pi_X : EX \to X$
and every perfect irreducible preimage of $EX$ is homeomorphic to EX; the map $\pi_X$ is then said to be a {\em projective cover of}\/ $X$. It is well-known that  the absolute is unique up to
homeomorphism, that every regular Hausdorff space has an absolute, and that a space $Y$ is an absolute of a regular Hausdorff space $X$ if, and only if, $Y$ is
an extremally disconnected Tychonoff space for which there exists a perfect
irreducible map $p_X : Y \to X$ (see, for example, \cite{ArP,Pon,Pon1,Pon2,Pon3,PS,PW}).

 It is well known
(see, for example, \cite{Wa})
that if $X$ is a compact Hausdorff space then $EX$ may be constructed as the Stone dual of the complete Boolean algebra ${\sf RC}(X)$, with $\pi_X$ the map
$EX={\sf Ult}({\sf RC}(X))\to X$
that assigns to an ultrafilter $\mathfrak u$ of regular closed sets of $X$ the only point in $\bigcap\mathfrak u$.
 Further, the absolute $EX$ of a Tychonoff space $X$ with its map $\pi_X:EX\lra X$ can be constructed by pulling the map $\pi_{\beta X}$ of the absolute $E(\beta X)$ of the Stone-\v{C}ech compactification $\b X$ of $X$ along the embedding $X\lra\beta X$.
 If $X$ is a locally compact Hausdorff space, then, as one easily sees, in the above construction of the absolute of $X$ one may replace $\b X$ by $\a X$ -- the Alexandroff (one-point) compactification of $X$.

For the definition and properties of absolutes of arbitrary topological spaces,  see the survey paper \cite{PS}. A modern and very efficient presentation appeared in \cite{Rump}.

We will  make essential use of {\em Alexandroff's Theorem} \cite[p. 346]{Alex}
(see also \cite[Theorem (d) (3), p. 455]{PW}),
which one derives easily from Ponomarev's results \cite{Pon1} on irreducible maps:
{\em For $p:X\longrightarrow Y$ a closed irreducible map of topological spaces, the map
  $$\rho_p:\operatorname{RC}(X)\longrightarrow \operatorname{RC}(Y),\ \ H\mapsto p(H),$$
  is a Boolean isomorphism, with $\rho_p^{-1}(K)=\mbox{{\rm cl}}_X(p^{-1}(\mbox{{\rm int}}_Y(K)))$, for all $K\in\operatorname{RC}(Y)$.}

  We will also use the well-known fact
(see, e.g., \cite{CNG},
p.271, and, for a proof, \cite{VDDB})
that, {\em when $Y$ is a dense subspace of a topological space $X$, then one has the Boolean isomorphisms $r$ and $e$ that are inverse to each other:}
 $$r:{\rm RC}(X)\longrightarrow {\rm RC}(Y),\  F\mapsto F\cap Y,\qquad{\rm and}\qquad e:{\rm RC}(Y)\longrightarrow{\rm RC}(X),\  G\mapsto \mbox{{\rm cl}}_X(G).$$
 \end{facts}

\section{A new duality theorem for the category $\HLC$}\label{s4}

\begin{nist}\label{P1-3}
\rm
Given a dual equivalence
$\xymatrix{{\AA}^{\rm op}\ar@/^0.6pc/[rr]^{T} & {\scriptstyle \ep\;\simeq\;\eta} & {\XX}\ar@/^0.6pc/[ll]^{S}}$
and an embedding $J$ of $\XX$ as a full subcategory of a category $\YY$,
the general categorical extension theorem  of \cite{DDT1} provides a {\em natural construction}\/ for a category $\BB$ into which $\AA$ may be fully embedded via $I$, and which allows for an extension along $I$ and $J$ of the dual equivalence between $\AA$ and $\XX$ to a dual equivalence between  $\BB$ and $\YY$. The construction depends on a given ${\cal X}$-{\em covering class} $\PP$ of morphisms in ${\cal Y}$, defined to
 satisfy the following conditions:
\begin{enumerate}
\renewcommand{\theenumi}{P\arabic{enumi}}
\item $1_X\in{\cal P}$ for every $\XX$-object $X$;
\item every ${\cal X}$-object $X$ is ${\cal P}$-projective in ${\cal Y}$, that is:\\ $\forall\;(p:Y\to Y')\in{\cal P},\,f:X\to Y'\ \exists\;g:X\to Y:\,p\circ g=f;$
\item ${\cal Y}$ has enough ${\cal P}$-projectives in ${\cal X}$, that is:\\
$\forall \,Y\in {\cal Y}\;\exists\,(p:X\to Y)\in{\cal P}\ :\ X\in {\cal X}.$
\end{enumerate}
Without loss of generality we may assume
that ${\cal P}$ be closed under pre-composition with isomorphisms in ${\cal X}$ and under post-composition with isomorphisms in $\mathcal Y$.
\end{nist}

\begin{constr}\label{const}{\rm \cite{DDT1}}
\rm
For any functor $T:{\mathcal A}^{\rm op}\to\mathcal X $ and an ${\cal X}$-covering class ${\cal P}$ in ${\cal Y}$, we form the (comma) category
${\sf C}({\cal A},{\cal P},{\cal X}),$
defined as follows:

\begin{itemize}
\item objects are pairs $(A,p)$ with $A$ in ${\cal A}$ and $p:TA\longrightarrow Y$ in the class ${\cal P}$;

\item  morphisms $(\varphi,f):(A,p)\longrightarrow(A',p\ap)$ in ${\sf C}({\cal A},{\cal P},{\cal X})$ are given by
morphisms $\varphi:A\longrightarrow A'$ in ${\cal A}$ and  $f:Y'\longrightarrow Y$ in ${\cal Y}$,
such that  $p\circ T\varphi=f\circ p\ap$:
\begin{center}
$\xymatrix{TA\ar[d]_{p} & TA'\ar[l]_{T\varphi}\ar[d]^{p\ap}\\
            Y & Y'\ar[l]^{f}}$
            \end{center}

\item    composition is as in ${\cal A}$ and ${\cal Y}$, so that $(\varphi,f)$ as above gets composed with
$(\varphi',f'):(A',p\ap)\longrightarrow(A'',p\ap)$ by the horizontal pasting of diagrams, that is,
$$(\varphi',f')\circ(\varphi,f)= (\varphi'\circ\varphi,f\circ f').$$

\item $(1_A,1_{{\rm cod}(p)})$ is
the identity morphism
of the object $(A,p)$ in ${\sf C}({\cal A},{\cal P},{\cal X})$.
\end{itemize}

\noindent On the hom-sets of ${\sf C}({\cal A},{\cal P},{\cal X})$
one defines a compatible equivalence relation by
$$(\varphi,f)\sim(\psi,g)\Longleftrightarrow f=g,$$
for all $(\varphi,f),(\psi,g):(A,p)\longrightarrow(A',p\ap).$
We let ${\cal B}$ be the quotient category
$${\cal B}={\sf C}({\cal A},{\cal P},{\cal X})/\!\sim,$$
{\em i.e.,} ${\cal B}$ has the same objects as ${\sf C}({\cal A},{\cal P},{\cal X})$, morphisms in $\cal B$ are the equivalence classes $[\p,f]$ of morphisms $(\p,f)$ in ${\sf C}({\cal A},{\cal P},{\cal X})$, and their composition proceeds by the composition of representatives.

\noindent One has the functor $I:{\cal A}\longrightarrow{\cal B}$, defined by
$$(\varphi:A\longrightarrow A')\mapsto (\;I\varphi=[\varphi,T\varphi]:(A, 1_{TA})\longrightarrow(A',1_{TA'})\;),$$
which, when $T$ is faithful, is a full embedding by P1, so that $\mathcal A$ may be considered as a full subcategory of $\mathcal B$. When $T$ is also full, by P2, morphisms $[\varphi,f]$ in $\mathcal B$ are fully determined by $f$, and when $T$ is a dual equivalence, as it will be the case henceforth, $\mathcal B$ becomes dually equivalent to $\mathcal Y$ by P3.
\end{constr}

\begin{theorem}\label{more duality}{\rm \cite{DDT1}}
Let ${\cal X}$ be a full subcategory of\/   ${\cal Y}$ and ${\cal P}$ an ${\cal X}$-covering class in ${\cal Y}$.
Then any dual equivalence $\xymatrix{{{\cal A}}^{\rm op}\ar@/^0.6pc/[rr]^{T} & {\scriptstyle \varepsilon\;\simeq\;\eta} & {{\cal X}}\ar@/^0.6pc/[ll]^{S}}$ can be extended to a dual equivalence $\xymatrix{{{\cal B}}^{\rm op}\ar@/^0.6pc/[rr]^{\tilde{T}} & {\scriptstyle \tilde{\varepsilon}\;\simeq\;\tilde{\eta}} & {{\cal Y}}\ar@/^0.6pc/[ll]^{\tilde{S}}}$ along $I$ and $J$, with $J:{\cal X}\hookrightarrow {\cal Y}$ the inclusion functor and $I$
the full embedding ${\cal A}\to{\cal B}:={\sf C}({\cal A},{\cal P},{\cal X})/\!\sim$,
being defined as above. The lifted dual equivalence may be chosen to satisfy $$\tilde{T}\tilde{S}={\rm Id}_{{\cal Y}},\quad
\tilde{\eta}=1_{{\rm Id}_{{\cal Y}}},\quad \tilde{T}\tilde{\varepsilon}=1_{\tilde{T}},\quad \tilde{\varepsilon}\tilde{S}=1_{\tilde{S}},$$
and the canonical isomorphism $\gamma:IS\to\tilde{S}J$ then satisfies  $\tilde{T}\gamma=J\eta$ and $\gamma T\circ I\varepsilon=\tilde{\varepsilon} I$.
\begin{center}
$\xymatrix{{\cal Y}\ar[r]_{\tilde{S}}\ar@/^1.0pc/[rr]^{\rm Id_{{\cal Y}}}\ar@{}[rd]|{\cong} & {\cal B}^{\rm op}\ar[r]_{\tilde{T}} & {\cal Y}\\
            {\cal X}\ar[r]_{S}\ar[u]^{J} & {\cal A}^{\rm op}\ar[r]_{T}\ar[u]_{I} & {\cal X}\ar[u]_{J}}$
            \end{center}
\end{theorem}

\begin{proof}{\em (Sketch)}
The functor $\tilde{T}$ may be taken to be the projection $[\varphi,f]\mapsto f$. Defining $\tilde{S}$ on objects, with P3  one chooses for every $Y\in |{\cal Y}|$ a morphism $\pi_Y:EY\longrightarrow Y$ in ${\cal P}$, with $\pi_X=1_X$ for all $X\in |{\cal X}|$ (according to P1), and then puts $\tilde{S}Y=(SEY,\pi_Y\circ\eta^{-1}_{EY})$. For a morphism $f:Y'\longrightarrow Y$ in ${\cal Y}$, again, P2 and the fullness of $T$ allow one to choose a morphism $\varphi_f:SEY\longrightarrow SEY'$ in ${\cal A}$ with
$\pi_Y\circ\eta^{-1}_{EY}\circ T\varphi_f=f\circ\pi_{Y'}\circ\eta^{-1}_{EY'}$ and to set
$\tilde{S}f=[\varphi_f,f]$.
For an object $(A,p:TA\longrightarrow Y)$ in $\cal B$ one puts $\tilde{\varepsilon}_{(A,p)}=[\varphi_{(A,p)},1_Y]$, with any ${\cal A}$-morphism $\varphi_{(A,p)}:A\longrightarrow SEY$ satisfying $p\circ T\varphi_{(A,p)}=\pi_Y\circ\eta^{-1}_{EY}$. Then $\tilde{\varepsilon}$ is, like $\tilde{\eta}=1_{{\rm Id}_{{\cal Y}}}$, a natural isomorphism satisfying the claimed identities.
\end{proof}

We recall some definitions and results from \cite{DD} and \cite{DD1} that are needed for our exposition.

\begin{notas}\label{zalgnl}
\rm
For a Boolean algebra $A$, $X\subseteq  \Ult(A)$ and every $a\in A$ we set
$$\ep_A^X(a)\ensuremath{\overset{\mathrm{df}}{=}} X\cap \ep_A(a),$$ thus defining the map $\ep_A^X: A\longrightarrow  {\sf P}(X)$ into the power set of $X$.
Since $\ep_A^X(A)\sbe \co(X)$, where $X$ is regarded as a subspace of $\Ultsf(A)$,  we may restrict this map to obtain the map $$\bar{\ep}_A^X:A\to \co(X).$$
For a topological space $X$
 we put
$\UP_X\df\{\ult_x\st x\in X\}$
(see Facts \ref{Stonel} for the notation $\ult_x$)
and consider the map
$$\up_{X}:X\longrightarrow  \UP_X,\ \ x\mapsto \ult_x,$$
noting that  {\em if $X$ is a zero-dimensional Hausdorff space, then $\up_{X}$ is a bijection.}
\end{notas}

A pair $(A,X)$, with a Boolean algebra $A$
and $X\subseteq \Ult(A)$,
is called a {\em Boolean z-algebra} (briefly, a {\em
z-algebra}) (\cite{DD}),  if for every $0\neq a\in A$ there exists
$\ult\in X$ such that $a\in\ult$. A z-algebra $(A,X)$ is said to be {\em complete} (\cite{DD})   if $A$ is complete as a Boolean algebra.
Obviously, a pair $(A,X)$ is a z-algebra if, and only if,  $A$ is a Boolean algebra and $X$ is
 a dense subset of\/ ${\sf Ult} (A)$.

The category $\ZCB$ (as introduced in \cite{DD}) has as objects complete z-algebras,  and
its morphisms $(\varphi,f):(A,X)\to(A^{\prime},X^{\prime})$
are given by a Boolean homomorphism
$\varphi:A\to A^{\prime}$ and the {\bf Set}-map $f:X^{\prime}\to X$ which is simply a restriction of $\Ultsf(\p)$;
the composition with $(\p\ap,f\ap)$ is defined as follows: $(\varphi^{\prime},f^{\prime})\circ (\varphi,f)=
(\varphi^{\prime}\circ\varphi, f\circ f^{\prime})$.

We may now recall the duality theorem for the category $\EDT$ of extremally disconnected Tychonoff spaces and continuous maps as obtained in \cite{DD}.

\begin{theorem}\label{zdualityedll}{\rm \cite{DD}}
The categories $\EDT$ and $\ZCB$ are dually equivalent.
\end{theorem}

\noindent{\it Proof (Sketch).} The dual equivalence is realized by the functors
$${\sf C}:\EDT\longrightarrow \ZCB^{\rm op}\ \ \mbox{ and }\ \ {\sf U}:\ZCB^{\rm op}\longrightarrow \EDT,$$
where for  every $X\in|\EDT|$,
${\sf C}(X)\ensuremath{\overset{\mathrm{df}}{=}}({\sf CO} (X),\UP_X);$
for  $f\in\EDT(X,Y)$,
 ${\sf C}(f)\ensuremath{\overset{\mathrm{df}}{=}} ({\sf CO} (f),\UP_f),$
 with
$\UP_f:\UP_X\longrightarrow  \UP_Y$  defined by
$\UP_f(\ult_x)\ensuremath{\overset{\mathrm{df}}{=}}\ult_{f(x)}$ for every $x\in X$;
for every $(A,X)\in|\ZCB|$,  ${\sf U}(A,X)\ensuremath{\overset{\mathrm{df}}{=}} X$ (where $X$ is regarded as a subspace of
${\sf Ult} (A)$);
if
$(\varphi,f)$
is a $\ZCB$-morphism, then
${\sf U}(\varphi,f)\ensuremath{\overset{\mathrm{df}}{=}} f.$
Further, we have the natural isomorphism $\ep\ap: \Id_{\ZCB}\lra \CCC\circ \UUU,\ \ (A,X)\mapsto \ep\ap_{(A,X)}$, where $\ep\ap_{(A,X)}\df (\bar{\ep}_A^X,\ups_X\inv)$, and the natural isomorphism $\eta\ap: \Id_{\EDT}\lra \UUU\circ \CCC,\ \ X\mapsto \eta\ap_{X}$, where $\eta\ap_{X}\df \ups_X$.
\sqs

We also need to recall the following notion
from  \cite{DD1}.
A  z-algebra $(A,X)$ is called a {\em Boolean lz-algebra}
(briefly, an {\em lz-algebra})
if for every $\ult\in X$ there exists $a\in A$ such that $\ult\in \ep_A(a)\subseteq  X$.
Obviously, a z-algebra $(A,X)$ is an lz-algebra if, and only if,  $X$ is
 an open subset of\/ ${\sf Ult} (A)$.

 Let $\EDL$ be the category of extremally disconnected locally compact Hausdorff spaces and continuous maps,  and $\LZCB$ be the full subcategory of the category $\ZCB$ having as objects all complete lz-algebras. Then, using Theorem \ref{zdualityedll} and some results from \cite{DD1}, one obtains the following assertion:

\begin{theorem}\label{zdualityed1}
The categories $\EDL$ and $\LZCB$ are dually equivalent.
\end{theorem}

This dual equivalence is realized by the respective restrictions of the  functors ${\sf C}$ and ${\sf U}$ described above. For simplicity, the restrictions will be denoted again by ${\sf C}$ and ${\sf U}$.

The next assertion of Rump \cite[Proposition 14]{Rump} will be very useful for us. We denote by $Pf$ the class of perfect surjections in $\Top$ and use the (standard) notion
of  {\em $Pf$-projective object} as appearing in
condition (P2) of  \ref{P1-3} with $\XX=\YY=\Top$ and $\PP=Pf$.

\begin{pro}\label{rumped}{\rm \cite{Rump}}
A space $X$  is
$Pf$-projective in $\Top$ if, and only if, $X$ is extremally disconnected.
\end{pro}

Now we are ready to apply our general categorical theorem for extensions of dualities to the concrete case of $\AA={\bf lzCBoo}$, $\XX={\bf EdLKH}$, $\YY={\bf LKHaus}$ and the class of perfect irreducible mappings between locally compact Hausdorff spaces with extremally disconnected domain in the role of $\PP$.

Since local compactness is an inverse invariant of perfect mappings (see, for example, \cite[Theorem 3.7.24]{E}), one obtains (using Facts \ref{projectivel}) that the absolute of a locally compact Hausdorff space is an extremally disconnected locally compact Hausdorff space.
 With Facts \ref{projectivel}  one easily sees that the class ${\cal P}$
 satisfies conditions P1 and P3 of \ref{P1-3}, and using Proposition \ref{rumped} one obtains that $\PP$ satisfies condition P2
  as well. Thus, $\PP$ is an \EDL-covering class in {\bf LKHaus}.  An application of Theorem \ref{more duality} to the duality
\begin{center}
$\xymatrix{{\cal A}^{\rm op}=\LZCB^{\rm op}\ar@/^0.6pc/[rr]^{\quad T\,=\,{\sf U}} & {\scriptstyle \simeq} & \EDL={\cal X}\ar@/^0.6pc/[ll]^{\quad S\,=\,{\sf C}}
}$
\end{center}
 described in Theorem \ref{zdualityed1} produces the following dual representation of {\bf LKHaus}:

\begin{pro}\label{Fed applicationl}
$\xymatrix{{\cal B}^{\rm op}:=({\sf C}({\cal A},{\cal P},{\cal X})/\!\sim)^{\rm op}\ar@/^0.6pc/[rr]^{\quad\quad\tilde{T}} & {\scriptstyle \simeq} & \quad{{\bf LKHaus}=:{\cal Y}}\ar@/^0.6pc/[ll]^{\quad\quad\tilde{S}}.
}$
\end{pro}

\noindent Here the objects of the category ${\sf C}({\cal A},{\cal P},{\cal X})$ are triples $(A,Z,p\!:\!Z\to Y)$ with $p\in{\cal P}$ and  $(A,Z)\in|\LZCB|$ (that is, $A$ is a complete Boolean algebra and $Z$ is an open dense subset of $\Ultsf(A)$); morphisms $(\varphi,g,f):(A,Z,p)\to(A',Z\ap,p\ap)$ are given by Boolean homomorphisms $\varphi:A\to A'$,  maps $g:Z\ap\to Z$ that are restrictions of the maps $\Ultsf(\p):\Ultsf(A\ap)\to\Ultsf(A)$, and maps $f$ that make the diagram
\begin{center}
$\xymatrix{Z\ar[d]_{p} & Z\ap\ar[l]_{g}\ar[d]^{p\ap}\\
            Y & Y'\ar[l]^{f}}$
            \end{center}
commute (see Construction \ref{const}). We note that such a map $f$ is necessarily continuous and uniquely determined by $g$ since, as a closed surjection, the map $p\ap$ in ${\cal P}$ provides its codomain with the quotient topology of its domain. The projection functor ${\sf C}({\cal A},{\cal P},{\cal X})^{\rm op}\to {\bf LKHaus},\;(\varphi,g,f)\mapsto f,$ induces the compatible relation $\sim$ on ${\sf C}({\cal A},{\cal P},{\cal X})$, so that
$((\varphi,g,f)\sim(\p_1,g_1,f_1)\Longleftrightarrow f=f_1)$
for $(\varphi,g,f),(\p_1,g_1,f_1):(A,Z,p)\to(A',Z\ap,p\ap)$.
 We obtain the quotient category ${\cal B}$, with the same objects as in ${\sf C}({\cal A},{\cal P},{\cal X})$. The contravariant functor  $\tilde{T}$ is induced by the projection functor; that is:
 $$\tilde{T}:{\cal B}^{\rm op}\to{\bf LKHaus},\;([\varphi,g,f]:(A,Z,p)\to(A',Z\ap,p\ap))\mapsto f.$$
 With $\pi_Y: EY\to Y$ denoting the projective cover of a locally compact Hausdorff space $Y$ (see Facts \ref{projectivel}), the adjoint $\tilde{S}$ of $\tilde{T}$, as defined  in Theorem \ref{more duality}, formally assigns to $Y$
  the ${\cal B}$-object $({\sf C}(EY), \pi_Y\circ(\eta\ap_{EY})\inv)$ which, however, is naturally isomorphic to $({\sf RC}(\a Y),EY,\pi_Y)$; here $\eta\ap$ and $\varepsilon\ap$ are, respectively, the units and counit of the dual equivalence from Theorem  \ref{zdualityed1}:
 \begin{center}
 $\xymatrix{EY={\sf U(RC}(\a Y),EY)\ar[dd]_{\pi_Y} &  & & {\sf U(C}(EY))\ar[d]^{(\eta\ap_{EY})\inv}\ar[lll]_{\quad{\sf U}(\varepsilon\ap_{({\sf RC}(\a Y), EY)})}\\
& & & EY\ar[d]^{\pi_Y}  \\
 Y & & & Y\ar[lll]_{1_Y}
 	}$
 \end{center}
 We may therefore assume $\tilde{S}(Y)=({\sf RC}(\a Y),EY,\pi_Y)$
 and $\tilde{S}(f)=[(\varphi,g)_f,f]$ for $f:Y^{\,\prime}\to Y$ in ${\bf LKHaus}$, where $(\varphi,g)_f:({\sf RC}(\a Y),EY)\to ({\sf RC}(\a Y^{\,\prime}),EY\ap)$ is an $\LZCB$-morphism such that
 $\pi_Y\circ {\sf U}((\varphi,g)_f)=f\circ\pi_{Y^{\,\prime}}$.
 We may even assume that $\tilde{S}(Y)=({\sf RC}(Y),E\ap Y,\pi\ap_Y)$, where $E\ap Y=\Ultsf(e)(EY)$, $e:\rc(Y)\to\rc(\a Y)$ is the Boolean isomorphism described at the end of Facts \ref{projectivel}, and $\pi\ap_Y=\pi_Y\circ ((\Ultsf(e))\inv)|E\ap Y$. Accordingly, for  $f:Y^{\,\prime}\to Y$ in ${\bf LKHaus}$, we may assume that $\tilde{S}(f)=[(\varphi\ap,g\ap)_f,f]$, where $(\varphi\ap,g\ap)_f:({\sf RC}(Y),E\ap Y)\to ({\sf RC}(Y^{\,\prime}),E\ap Y\ap)$ is an $\LZCB$-morphism such that
 $\pi\ap_Y\circ {\sf U}((\varphi\ap,g\ap)_f)=f\circ\pi\ap_{Y^{\,\prime}}$. For simplicity, we write  $EY$ instead of $E\ap Y$ and  $\pi_Y$
 instead of $\pi\ap_Y$.
 This
  leaves all assertions of Theorem \ref{more duality} in tact. In fact, it simplifies them, since we now have that the natural isomorphism $\gamma$ of Theorem \ref{more duality} is actually an identity transformation: $IS=\tilde{S}J$.

 \begin{fact}\label{creql}
 Let $(A,Z,p:Z\to Y)$ be a $\CAPX$-object. Then, for the map $\rho_p$ of Facts \ref{projectivel}  and
  $\KO(Z)$ denoting the set of
 compact open subsets of $Z$, one has
 $\rho_p(\KO(Z))=\CR(Y)$.
  \end{fact}

\doc Since $Z$ is extremally disconnected, we have  $\KO(Z)=\CR(Z)$ and must  therefore show $\rho_p(\CR(Z))=\CR(Y)$. The inclusion  $\rho_p(\CR(Z))\sbe\CR(Y)$ is obvious. Conversely, for $G\in \CR(Y)$, since $p$ is perfect, \cite[Theorem 3.7.2]{E} implies that $p\inv(G)$ is compact. Then $\rho_p\inv(G)=\cl_Z(p\inv(\int_Y(G)))\sbe p\inv(G)$ and thus $\rho_p\inv(G)\in\CR(Z)$.
\sqs

\begin{pro}\label{smfpll}
 Let $(A,Z,p:Z\to Y)$ be a $\CAPX$-object. With $\smallfrown_p$ defined by
 $$a\smallfrown_p b\iff p(\varepsilon_A^Z(a))\cap p(\varepsilon_A^Z(b))\neq \emptyset\quad(a,b\in A)$$
 and $ \BBBB_{p}\df (\ep_A^Z)\inv(\KO(Z))$,
 $(A,\smf_p,\BBBB_{p})$ becomes a complete local contact algebra and
 $$\rho_p\circ\bar{\ep}_A^Z:(A,\smf_p,\BBBB_{p})\lra (\rc(Y),\smf_Y,\CR(Y))$$  a $\CLCA$-isomorphism. Also, $\BBBB_{p}=\{a\in A\st \ep_A(a)\sbe Z\}$.
 \end{pro}

  \doc By Proposition \ref{stanlocn}(a), $(\rc(Y),\smf_Y,\CR(Y))$ is a complete local contact algebra.  Fact \ref{creql} gives us the equality $(\rho_p\circ\bar{\ep}_A^Z)(\BBBB_{p})=\CR(Y)$. Further,  by definition, the relation $\smf_p$ is obtained by transferring the contact relation  $\smf_Y$ along the (inverse of the) isomorphism $\rho_p\circ\bar{\varepsilon}_A^Z$ (note that, since $Z$ is extremally disconneceted, $\RC(Z)=\CO(Z)$).
  Hence,  $\rho_p\circ\bar{\ep}_A^Z$ is an $\CLCA$-isomorphism.  For proving the last assertion,  let $a\in A$ and $\ep_A(a)\sbe Z$. Then
 $\ep_A^Z(a)=\ep_A(a)\in\KO(Z)$.
 Conversely, if $\ep_A^Z(a)\in\KO(Z)$ then $\ep_A^Z(a)=\ep_A(a)$ (because $Z$ is dense in $\Ultsf(A)$) and, thus, $\ep_A(a)\sbe Z$. Therefore, $\BBBB_{p}=\{a\in A\st \ep_A(a)\sbe Z\}$.
 \sqs

\begin{rem}\label{pupv1l}
\rm For $(A,Z,p:Z\to Y)\in |{\sf C}({\cal A},{\cal P},{\cal X})|$ and all $\mathfrak{u},\mathfrak{v}\in Z$ we note that
$$\mathfrak{u}\smallfrown_p \mathfrak{v}\Longleftrightarrow p(\mathfrak{u})=p(\mathfrak{v}).$$
Indeed, the implication ``$\Longrightarrow$'' follows easily from the Hausdorffness of $Y$ and the fact that $\{\varepsilon_A^Z(a)\ |\ a\in A\}$ is an open base for $Z$, while  the converse implication is obvious.
\end{rem}

\begin{defi}\label{dboo}
\rm
Let ${\bf DBoo}$ be the category
whose objects are
all complete local contact algebras, and whose morphisms $\varphi: (A_1,\smf_1,\BBBB_1)\lra (A_2,\smf_2,\BBBB_2)$ are  Boolean homomorphism between $A_1$ and $A_2$ reflecting the contact relation (or, equivalently, preserving the associated relation $\ll$) and satisfying condition (CLC4) from Definition \ref{dhc};
the composition of ${\bf DBoo}$-morphisms proceeds by ordinary map composition.
\end{defi}

Note that the objects of the category $\DBoo$ are the same as those of the category {\bf CLCA} but, unlike $\CLCA$-morphisms, its morphisms are {\em Boolean} homomorphism respecting in a natural way the contact relations and boundedness; moreover, their categorical composition coincides with their composition in $\Set$.
\begin{pro}\label{U}
There is a functor
\begin{center}
$\xymatrix{{\sf C}({\cal A},{\cal P},{\cal X})\ar[r]^{\;\;U} & {\bf DBoo},}$
\end{center}
defined on objects by $U(A,Z,p:Z\to Y)=(A,\smf_p,\BBBB_{p})$,  and on morphisms by $U(\p,g,f)=\p$.
\end{pro}

\doc  Proposition \ref{smfpll} shows that $U$ is well-defined on objects.
For proving that $U$ is well-defined on morphisms,
 we show that, for a $\CAPX$-morrphism $(\p,g,f):(A,Z,p)\lra (A\ap,Z\ap,p\ap),$
 the map
 $\varphi:A\to A'$ reflects the contact relations imposed by $U$ and satisfies condition (CLC4).  So, let $a,b\in A$ and  $\varphi(a)\smallfrown_{p\ap}\varphi(b)$. Then $p\ap(\mathfrak u')=p\ap(\mathfrak v')$ for some $\mathfrak u', \mathfrak v'\in Z\ap$ with $\varphi(a)\in\mathfrak u', \varphi(b)\in\mathfrak v'$; consequently,
$p(\varphi^{-1}(\mathfrak u'))=p(g(\ult\ap))=f(p\ap(\mathfrak u'))=f(p\ap(\mathfrak v'))=p(g(\ultv\ap))=p(\varphi^{-1}(\mathfrak v')),$
which implies $a\smallfrown_p b$. Hence, $\p$ reflects contact relations. Let now $b\ap\in\BBBB_{p\ap}$. Then $\ep_{A\ap}(b\ap)\sbe Z\ap$. Since $g(Z\ap)\sbe Z$, we obtain that $\Ultsf(\p)(\ep_{A\ap}(b\ap))=g(\ep_{A\ap}(b\ap))\sbe Z$ and $g(\ep_{A\ap}(b\ap))$ is compact. Thus there exists $F\in\KO(Z)$ such that $g(\ep_{A\ap}(b\ap))\sbe F$. Then, by the definition of $\BBBB_{p}$   (see Proposition \ref{smfpll}), $F=\ep_A(b)$ for some $b\in \BBBB_{p}$. So, we have that $\Ultsf(\p)(\ep_{A\ap}(b\ap))\sbe\ep_A(b)\sbe Z$. Thus, by the Stone Duality, $\ep_{A\ap}(b\ap)\sbe (\Ultsf(\p))\inv(\ep_A(b))=\co(\Ultsf(\p))(\ep_A(b))=\ep_{A\ap}(\p(b))$. This implies $b\ap\le\p(b)$. Therefore, $U$ is well-defined on morphisms. It is now easy to see that $U$ is a functor.
\sqs

\begin{pro}\label{eqrel}
Let $(A,\smf,\BBBB)$ be a local contact algebra. Then, letting  $Z$ denote the set\/ $\BUlt(A)$ regarded  as a subspace of $\Ultsf(A)$,
one has:

\smallskip

\noindent{\rm (a)} the relation $\smf$ on the ultrafilters in $A$ as defined in Facts \ref{deVriesl} is an equivalence relation on the set\/ $\BUlt(A)$;

\smallskip

\noindent{\rm (b)} for every $b\in A$ one has $b\in\BBBB$ if, and only if, $\ep_A(b)\sbe Z$;

\smallskip

\noindent{\rm (c)} $Z$ is an open dense subset of the space $\Ultsf(A)$ and, thus, $Z$ is a locally compact Hausdorff space; furthermore, if $A$ is complete, then $Z$ is extremally disconnected;

\smallskip

\noindent{\rm (d)} the natural quotient map $p_{A}:Z\to Y=Z/\smf$ is  perfect irreducible, and the quotient space $Y$ is a locally compact Hausdorff space.
\end{pro}

\doc (a) It is clear that the relation $\smf$ on $Z$ is reflexive and symmetric. To show its transitivity, we consider $\ult,\ultv,\ultw\in Z$ with $\ult\smf\ultv$ and $\ultv\smf\ultw$
and suppose that  there exist $a_0\in\ult$ and $c_0\in\ultw$ with $a_0\nsmf c_0$. Since $\ult$ is bounded, there exists $a_1\in\ult\cap\BBBB$. Then $a\df a_0\we a_1\in\ult\cap\BBBB$ and $a\nsmf c_0$, i.e., $a\ll c_0^*$. Consequently, there is $b\in\BBBB$ with $a\ll b\ll c_0^*$, and $a\nsmf b^*$ and $b\nsmf c_0$ follows. Then $b\nin\ultv$ and, thus, $b^*\in\ultv$. Therefore $a\smf b^*$, a contradiction. Hence, $\ult\smf\ultw$, and the proof of (a) is complete.

\medskip

\noindent(b) For $b\in \BBBB$ one clearly has $\ep_A(b)\sbe Z$. Conversely, for $b\in A$ and $\ep_A(b)\sbe Z$, assume that we had $b\nin\BBBB$. Then $b\in A\stm\BBBB=\clu_\infty$. Since, by Lemma \ref{neogrn}, $\clu_\infty$ is a cluster in $A$, there exists an ultrafilter $\ult$ in $A$ such that $b\in\ult\sbe A\stm \BBBB$. Then $\ult\in\ep_A(b)\stm Z$, a contradiction. Therefore $b\in\BBBB$. So, $b\in\BBBB$ if, and only if, $\ep_A(b)\sbe Z$.

\medskip

\noindent(c) For every $\ult\in Z$ there exists $b\in\ult\cap\BBBB$. Then, by (b), $\ult\in\ep_A(b)\sbe Z$, so that $Z$ is open in the space $\Ultsf(A)$. For proving its density
we let $0\neq a\in A$. By (BC3), there exists $0\neq b\in\BBBB$ with $b\ll a$. Then $b\le a$, $\ems\neq\ep_A(b)\sbe\ep_A(a)$, and $\ep_A(b)\sbe Z$ by (b). Thus, $\ep_A(a)\cap Z\nes$, which proves the density of $Z$ in $\Ultsf(A)$.

\medskip

\noindent(d) Writing $p\df p_{A}$ we first show that the map $p$ is irreducible. Suppose that there exists a closed proper subset $F$ of $Z$ such that $p(F)=Y$. Then there exists a closed subset $G$ of $\Ultsf(A)$ such that $G\cap Z=F$. Let $\ult\in Z\stm F$. Since $G$ is compact, there exists $a\in A$ such that $G\sbe \ep_A(a)$ and $\ult\nin\ep_A(a)$. Then $a\neq 1$ and, thus, $a^*\neq 0$. By (BC3), there exists $0\neq b\in\BBBB$ with $b\ll a^*$, i.e., $b\nsmf a$. There exists $\ultv\in\Ult(A)$ such that $b\in\ultv$, and $\ultv\in Z$ follows. Also, $a\nin\ultv$, i.e., $\ultv\nin\ep_A(a)$, which implies $\ultv\in Z\stm F$. Since $p(F)=Y$, there exists $\ultw\in F$ with $p(\ultv)=p(\ultw)$. Then $\ultv\smf\ultw$, and since $a\in\ultw$, we obtain $a\smf b$, a contradiction. Hence, the map $p$ is irreducible.

We now show that, for every $\ult\in Z$, the set $p\inv(p(\ult))$ is closed in $\Ultsf(A)$ and, thus, compact. Note that $p\inv(p(\ult))=\{\ultv\in Z\st \ult\smf\ultv\}$, and let $\ultw\in \Ultsf(A)\stm p\inv(p(\ult))$. We first consider the case $\ult\nsmf\ultw$ (which certainly happens  when $\ultw\in Z$). Then there exist $a\in\ult$ and $c\in\ultw$ such that $a\nsmf c$.
Obviously, $\ultw\in\ep_A(c)$. Suppose that there exists $\ultv\in\ep_A(c)\cap p\inv(p(\ult))$. Then $c\in\ultv$ and, hence, $a\smf c$ follows -- a contradiction. Therefore, $\ultw\in\ep_A(c)\sbe \Ultsf(A)\stm p\inv(p(\ult))$. Considering now the the case $\ult\smf\ultw$ we have  $\ultw\nin Z$. Thus $\ultw\sbe A\stm \BBBB$, and there exists $b_0\in\ult\cap\BBBB$. Since $b_0\ll 1$, (BC1) implies that there exists $b\in\BBBB$ with $b_0\ll b$, i.e., $b_0\nsmf b^*$. Since $b\nin\ultw$, we have $b^*\in\ultw$. Now, from $b_0\in\ult$ and $\ult\smf\ultw$, we obtain that $b_0\smf b^*$, a contradiction. Hence this case is impossible. We conclude that the set $\Ultsf(A)\stm p\inv(p(\ult))$ is open in $\Ultsf(A)$, i.e., the set $p\inv(p(\ult))$ is closed in $\Ultsf(A)$.

For proving that $p$ is perfect, it remains to be shown that $p$ is a closed map. By \cite[Proposition 2.4.9(iii)]{E}, we have to prove that for every open subset $U$ of $Z$, the union $W$ of all equivalence classes that are contained in $U$ is open in $Z$. So, let $U$ be an open subset of $Z$ and $\ult\in Z$ be such that $[\ult]\df p\inv(p(\ult))$ is a subset of $U$. As already shown above, $[\ult]$ is compact. Hence, using (b) and (c), we can find $b_0\in\BBBB$ such that $[\ult]\sbe\ep_A(b_0)\sbe U$, and we have $b_0\in\ult$. Putting
$$V= Z\stm \bigcup\{[\ultv]\st \ultv\in Z\cap\ep_A(b_0^*)\}$$ we obtain $V=Z\stm \bigcup\{[\ultv]\st b_0^*\in\ultv\in Z\}$. For showing $[\ult]\sbe V$ we consider $\ultv\in\ep_A(b_0^*)\cap Z$. Then $b_0\nin\ultv$ and, since $[\ult]\sbe\ep_A(b_0)$, we obtain that every $\ultw$ which belongs to $[\ult]$ contains $b_0$. Thus $\ultv\nin[\ult]$, and $[\ult]\cap[\ultv]=\ems$ follows. Consequently,  $[\ult]\sbe V$.

To prove that $V$ is open, we consider $\ultw\in V$ and suppose that, for every $b\in\ultw\cap\BBBB$, $\ep_A(b)\nsubseteq V$. Then, for every $b\in\ultw\cap\BBBB$, $\ep_A(b)\cap \bigcup\{[\ultv]\st b_0^*\in\ultv\in Z\}\nes$. Hence, for every $b\in\ultw\cap\BBBB$, there exist $\ult_b\in\ep_A(b)$ and $\ultv_b\in\ep_A(b_0^*)\cap Z$ such that $\ult_b\in[\ultv_b]$. This means that for every $b\in\ultw\cap\BBBB$, there exist $\ult_b,\ultv_b\in Z$ such that $b\in\ult_b$, $b_0^*\in\ultv_b$ and $\ult_b\smf\ultv_b$. Hence, for every $b\in\ultw\cap\BBBB$, we have $b_0^*\smf b$. Since $\ultw\cap\BBBB$ is a filter-base for $\ultw$, we obtain that $b_0^*\smf c$ for every $c\in\ultw$. Thus, by Facts \ref{deVriesl}, $b_0^*\in\clu_\ultw$, where $\clu_\ultw$ is a cluster in $A$ (see Definitions \ref{Alexprn} and \ref{boundcl}, and Lemma \ref{Alexprn1}). Using again Facts \ref{deVriesl}, we find a $\ultw\ap\in\Ult(A)$ with $b_0^*\in\ultw\ap\sbe\clu_\ultw$, which implies $\ultw\smf_{\rm\! Al}\ultw\ap$ and then $\ultw\ap\in Z$. Indeed, if $\ultw\ap\nin Z$ then $\ultw\ap\sbe A\stm \BBBB$. Since $\ultw\in Z$, there exists $b_1\in\ultw\cap\BBBB$. Then there exists $b_2\in\BBBB$ such that $b_1\ll_\smf b_2$ (i.e., $b_1\nsmf b_2^*$). Since $b_2\in\BBBB$, we obtain $b_2\nin\ultw\ap$ and, thus $b_2^*\in\ultw\ap$. Since $b_1\in\BBBB$, $b_1\nsmf_{Al} b_2^*$ follows. But this contradicts $b_1\in\ultw$ and $b_2^*\in\ultw\ap$. Hence, $\ultw\ap\in Z$, which implies that $\ultw\smf\ultw\ap$ and, thus,  $[\ultw]=[\ultw\ap]$. Since $\ultw\ap\in\ep_A(b_0^*)\cap Z$, we obtain that $\ultw\in Z\stm V$, a contradiction. Thus, $V$ is open in $Z$.

Finally, we show that $V$ is a subset of the union $W$ of all equivalence classes which are contained in $U$ and, to this end, we consider
$\ultw\in Z$ with $[\ultw]\cap V\nes$. Assume that we had $[\ultw]\nsubseteq U$; then, since $\ep_A(b_0)\sbe U$, we obtain that $[\ultw]\nsubseteq\ep_A(b_0)$. Hence, there exists $\ultw\ap\in[\ultw]\cap(Z\stm\ep_A(b_0))$. Then $\ultw\ap\in\ep_A(b_0^*)$ and, therefore, $b_0^*\in\ultw\ap\in Z$. Thus, $[\ultw]=[\ultw\ap]\sbe Z\stm V$, i.e., $[\ultw]\cap V=\ems$, a contradiction. Hence, $[\ultw]\sbe U$. So, we have proved that for every $\ult\in W$ there exists an open $V$ such that $\ult\in V\sbe W$. Therefore, $W$ is open. This implies that $p$ is a closed map, and we conclude that $p$ is a perfect irreducible map.

Since $Z$ is a locally compact Hausdorff space and $p:Z\to Y$ is perfect, \cite[Theorems 3.7.20, 3.7.21]{E} imply that $Y$ is locally compact Hausdorff as well.
\sqs

\begin{pro}\label{Ul equivalencel}
The functor $U$ has a right inverse $W:{\bf DBoo}\to {\sf C}({\cal A},{\cal P},{\cal X})$ with $W\circ U\cong \rm{Id}$. In particular, $U$ is an equivalence of categories.	
\end{pro}

\doc For $(A,\smf,\BBBB)\in|\DBoo|$ one defines
$W(A,\smf,\BBBB)= (A,Z,p),$
with $Z=\BUlt(A)$ and $p:Z\to Z/\smf$ the natural quotient map; as above, $\smf$ denotes the contact relation on $A$, as well as its extension to $\Ult(A)$. In Proposition \ref{eqrel} we proved that $\smf$ is an equivalence relation on $Z$, $p$ is an irreducible perfect map, and that $(A,Z)$ is in $\LZCB$. Hence, $W$ is well-defined on objects. Next we show $ U(W(A,\smf,\BBBB))=(A,\smf,\BBBB)$. Indeed, by Proposition \ref{U}, we have
$U(A,Z,p)=(A,\smf_p,\BBBB_{p})$,  and we must prove $\smf=\smf_p$ and $\BBBB=\BBBB_{p}$. For $a,b\in A$ we have
$a\smf_p b\iff (\ex \ult,\ultv\in Z$ with $a\in\ult,\ b\in\ultv$ and $p(\ult)=p(\ultv))\iff (\ex \ult,\ultv\in Z$ with $a\in\ult,\ b\in\ultv$ and $\ult\smf\ultv$). Hence, $a\smf_p b$ implies $a\smf b$. Conversely, if $a\smf b$, then, by (BC2), there exist $a_1\le a$ and $b_1\le b$ such that $a_1,b_1\in\BBBB$ and $a_1\smf b_1$. Then, by Facts \ref{deVriesl}, there exists $\ult,\ultv\in\Ult(A)$ such that $a_1\in\ult$, $b_1\in\ultv$ and $\ult\smf\ultv$. Clearly, this implies $\ult,\ultv\in Z$, $a\in\ult$, $b\in\ultv$,  $p(\ult)=p(\ultv)$ and, hence, $a\smf_p b$. We conclude that $\smf=\smf_p$. To prove that  $\BBBB=\BBBB_{p}$, using Propositions \ref{eqrel}(b) and \ref{smfpll} we obtain  $b\in\BBBB\iff \ep_A(b)\sbe Z\iff b\in \BBBB_{p}$. So, $U(W(A,\smf,\BBBB))=(A,\smf,\BBBB)$.

Computing $U\circ W$ on a $\DBoo$-morphism
$\p:A\to A\ap=(A\ap,\smf',\BBBB\ap)$
we write $W(A\ap)=(A\ap,Z\ap,p\ap)$ and first show that $\ult\ap\in Z\ap$ implies $\p\inv(\ult\ap)\in Z$.  Since $\ult\ap$ is bounded, there exists $b\ap\in\ult\ap\cap\BBBB\ap$. Then there exists $b\in\BBBB$ such that $b\ap\le\p(b)$. Thus $\p(b)\in\ult\ap$, which implies that $b\in\p\inv(\ult\ap)\cap\BBBB$. Hence, the ultrafilter $\p\inv(\ult\ap)$ is bounded, and we have $\Ultsf(\p)|Z\ap:Z\ap\to Z$.

We now show that, if $\ult\ap,\ultv\ap\in\BUlt(A\ap)$, then $\ult\ap\smf'\,\ultv\ap$ implies $\p\inv(\ult\ap)\smf\p\inv(\ultv\ap)$. Assuming that we had $\p\inv(\ult\ap)\nsmf\p\inv(\ultv\ap)$ we put $\ult\df\p\inv(\ult\ap)$ and
$\ultv\df\p\inv(\ultv\ap)$.  Then $p(\ult)\neq p(\ultv)$ and $p\ap(\ult\ap)= p\ap(\ultv\ap)$. Since the codomain $Y$ of $p$ is a Hausdorff space (see Proposition \ref{eqrel}(d)), there exist disjoint neighbourhoods $O_\ult,O_\ultv$ of, respectively, $p(\ult)$ and  $p(\ultv)$. Then there exist $a\in\ult$ and $b\in\ultv$ such that $\ult\in\ep_A^Z(a)\sbe p\inv(O_\ult)$ and $\ultv\in\ep_A^Z(b)\sbe p\inv(O_\ultv)$. Thus $p(\ep_A^Z(a))\cap p(\ep_A^Z(b))=\ems$,
which implies that $a\nsmf_p b$ and, therefore, $a\nsmf b$ (because $U(W(A,\smf,\BBBB))=(A,\smf,\BBBB)$). This implies  $\p(a)\nsmf\ap\p(b)$ and, since $\p(a)\in\ult\ap$, $\p(b)\in\ultv\ap$ and $\ult\ap\smf'\,\ultv\ap$, we obtain $\p(a)\smf'\,\p(b)$, a contradiction. Therefore, $\p\inv(\ult\ap)\smf\p\inv(\ultv\ap)$.
Now, with
$$Y'=Z\ap/\!\smallfrown\ap\mbox{ and }g_\p\df \Ultsf(\p)|Z\ap,$$ one obtains a uniquely determined continuous map $f_{\varphi}$ making the diagram
\begin{center}
$\xymatrix{Z\ar[d]_{p} & Z\ap\ar[l]_{g_\p}\ar[d]^{p\ap}\\
            Y & Y'\ar[l]_{f_{\varphi}}}$
            \end{center}
commute; explicitly,
$f_{\varphi}(p\ap(\mathfrak u'))=p(\varphi^{-1}(\mathfrak u'))$, for all $\mathfrak u'\in{\sf Ult}(A').$
Obviously then, with
$$W\varphi=(\varphi,g_\p,f_{\varphi}),$$ we obtain a well-defined functor $W$. Since, by the definition of the functor $U$, $U(W\varphi)=\varphi$ for all $\varphi$, one has $U\circ W={\rm Id}_{\bf DBoo}$.

Finally, we show $W\circ U\cong {\rm Id}_{\CAPX}$ and consider
 $(A,Z,p)\in|{\sf C}({\cal A},{\cal P},{\cal X})|$. Then $W(U(A,Z,p))=W(A,\smf_p,\BBBB_{p}) =(A,Z\ap,p\ap)$, where $$Z\ap=\{\ult\in\Ult(A)\st \ult\cap\BBBB_{p}\nes\}$$ and $p\ap:Z\ap\to Z\ap/\smf_p$ is the natural quotient mapping.
 We first show that $Z\equiv Z\ap$ and consider $\ult\in Z$. Then there exists $a\in\ult$ with $\ult\in\ep_A(a)\sbe Z$ since $Z$ is open in $\Ultsf(A)$.
 By Proposition \ref{smfpll}, $a\in \BBBB_{p}$ and thus $\ult\in Z\ap$, which gives $Z\sbe Z\ap$. Conversely, for $\ult\in Z\ap$ one has $\ult\cap \BBBB_{p}\nes$ and, therefore, obtains an element $a\in \ult\cap \BBBB_{p}$. Hence, $\ult\in\ep_A(a)\sbe Z$ and, thus, $\ult\in Z$. This shows $Z\equiv Z\ap.$
 Now, denoting by $Y$ the codomain of $p$ and writing $Y\ap$ for $Z/\smf_p$, one obtains a homeomorphism $h_p$ making the diagram
 \begin{center}
$\xymatrix{Z\ar[d]_{p\ap} & Z\ar[l]_{\quad 1_Z}\ar[d]^{p}\\
           Y\ap  & Y\ar[l]_{\quad h_{p}}}$
            \end{center}
commute. Indeed,  for every $\ult\in Z$, set $[\ult]=(p\ap)\inv(p\ap(\ult))$; then put $h_p(p(\ult))\df [\ult]$. The mapping $h_p$ is well-defined because, by Remark \ref{pupv1l}, for every $\ult,\ultv\in Z$, $p(\ult)= p(\ultv)\iff \ult\smf_p\ultv$, i.e., $p(\ult)= p(\ultv)\iff [\ult]=[\ultv]$. Also, defining $h_p\ap:Y\ap\to Y$ by $h_p\ap([\ult])\df p(\ult)$ for every $\ult\in Z$, we see that $h_p\ap$ is well defined and is inverse to $h_p$. Clearly, $h_p$ and $h_p\ap$ are continuous. Thus $h_p$ is a homeomorphism.
  Now we may define
$$\lambda_{(A,Z,p)}= (1_A,1_Z,h_p\inv)
  : (A,Z,p)\to W(U(A,Z,p))$$
and, noting that $h_p$ is a homeomorphism, obtain that  $\lambda_{(A,Z,p)}$ is an isomorphism in ${\sf C}({\cal A},{\cal P},{\cal X})$. One easily sees that the isomorphism is natural in $(A,Z,p)$, so that  $W\circ U\cong Id_{\CAPX}$ follows.
\sqs

\begin{lm}\label{U2l}
For all morphisms $(\varphi,g,f),(\psi,g_1,f_1):(A,Z,p)\to(A_1,Z_1,p_1)$ in ${\sf C}({\cal A},{\cal P},{\cal X})$, one has $f=f_1$ if, and only if, $(U(\varphi,g,f))\cuk=(U(\psi,g_1,f_1))\cuk$ (see Definition \ref{dhc} for the notation $h\cuk$).
 \end{lm}

\doc Since $U(\varphi,g,f)=\p$, we will simply write $\p\cuk$ instead of $(U(\varphi,g,f))\cuk$ (and the same for $\psi\cuk$).

 Under the hypothesis $f=f_1$ one has $p\circ g=f\circ \psi=f_1\circ \psi=p\circ g_1$, which means that, for all $\mathfrak u'\in Z_1$, $p(\p\inv(\ult\ap))=
p(\psi\inv(\ult\ap))$ and thus
$\varphi^{-1}(\mathfrak u')\smallfrown_p\psi^{-1}(\mathfrak u')$ (in the notation of Proposition \ref{smfpll}). To show $\varphi\cuk(a)\leq\psi\cuk(a)$ for $a\in A$, it suffices to prove that if
$b\in\BBBB_{p}$ and
$b\ll_p a$, then there is some $c\in\BBBB_{p}$ with $c\ll_{p} a$ and $\varphi(b)\leq\psi(c)$. Assuming the opposite, we obtain that there exists  $b_0\in\BBBB_{p}$ with
$b_0\ll_p a$ such that every $c\in\BBBB_{p}$ with $c\ll_{p} a$ satisfies  $\varphi(b_0)\nleq\psi(c)$. There exists
some $c\in \BBBB_{p}$ with $b_0\ll_p c\ll_p a$. Then $\varphi(b_0)\not\leq\psi(c)$. This implies $d\df\varphi(b_0)\wedge\psi(c^*)>0$. Then, by (BC3), there exists  $b_1\in \BBBB_{p_1}$ such that $b_1\le d$. Now
we find $\mathfrak u'\in {\Ult}(A_1)$ containing $b_1$. Then  $d\in\ult\ap$ and $\ult\ap\in Z_1$ (see the equality $``Z=Z\ap$" in the last paragraph of the proof of Proposition \ref{Ul equivalencel}). Consequently, $b_0\in\varphi^{-1}({\mathfrak u'})$ and $c^*\in\psi^{-1}({\mathfrak u'})$. But then $b_0\smallfrown_p c^*$, which means $b\not\ll_p c$, a contradiction. Thus  $\varphi\cuk(a)\leq\psi\cuk(a)$ for all $a\in A$.  Since $\psi\cuk(a)\leq\varphi\cuk(a)$ follows by symmetry, the identity $\varphi\cuk=\psi\cuk$ is confirmed.

Conversely, let $\varphi\cuk=\psi\cuk$. Since $p_1$ is surjective, to prove $f=f_1$, it suffices to show  $p\circ g=p\circ g_1$, i.e., that $\varphi^{-1}(\mathfrak u')\smallfrown_p\psi^{-1}(\mathfrak u')$ for all $\mathfrak u'\in Z_1$
 (see Remark \ref{pupv1l}).
 Assuming $\varphi^{-1}(\mathfrak u')\not\smallfrown_p\psi^{-1}(\mathfrak u')$ for some $\mathfrak u'\in Z_1$, we obtain $b\in\varphi^{-1}(\mathfrak u')$ and $a\in A$ with  $a^*\in\psi^{-1}(\mathfrak u')$ and $b\not\smallfrown_p a^*$, which means $\varphi(b),\psi(a^*)\in\mathfrak u'$ and $b\ll_p a$. Since $\ult\ap\in Z_1$
 and $Z_1=\{\ult\in\Ultsf(A_1)\st \ult\cap\BBBB_{p_1}\nes\}$
   (see again the proof of Proposition \ref{Ul equivalencel}), we find a $b\ap\in\ult\ap\cap\BBBB_{p_1}$. Then there exists $b_0\in\BBBB_{p}$ such that $b\ap\le\p(b_0)$. Thus $\p(b_0)\in\ult\ap$ and, hence, $b_0\in\p\inv(\ult\ap)$. Therefore, $b_2\df b_0\we b\in\BBBB_{p}\cap\p\inv(\ult\ap)$ and $b_2 \ll_p a$.
 Further, with the monotonicity of $\psi$ one has $\psi(a)\ge\bigvee\{\psi(c)\st c\in\BBBB_{p}, c\ll_p a\}=\psi\cuk(a)$ and, thus, $\psi(a^*)\le (\psi\cuk(a))^*$. Since $\varphi(b_2),\psi(a^*)\in\mathfrak u'$, we obtain $0\neq\p(b_2)\we\psi(a^*)\le\p(b_2)\we  (\psi\cuk(a))^*$. Therefore, $\p(b_2)\nleq\psi\cuk(a)$, and since $\p(b_2)\le \p\cuk(a)=\psi\cuk(a)$, we have  a contradiction. Thus, $f=f_1$.
 \sqs

\begin{constr}\label{quotientfunctors1}
\rm
With the compatible equivalence relation $\sim$ on ${\sf C}({\cal A},{\cal P},{\cal X})$ as defined in Proposition \ref{Fed applicationl}, and with $\backsim$ denoting the equivalence relation on {\bf DBoo} defined by  $\varphi\backsim\psi\iff \varphi\cuk=\psi\cuk$), the assertion of Lemma \ref{U2l} reads as $$(\varphi,g,f)\sim(\psi,g_1,f_1)\iff U(\varphi,g,f))\backsim U(\psi,g_1,f_1).$$ Denoting by $\langle\varphi\rangle$ the $\backsim$-equivalence class of a morphism $\varphi$ in {\bf DBoo}, the functor $U$  therefore induces the faithful functor
$${\cal B}=\xymatrix{{\sf C}({\cal A},{\cal P},{\cal X})/\!\sim\ar[r]^{\quad\overline{U}} & {\bf DBoo/\!\!\backsim}},$$
mapping objects like $U$, and morphisms by $[\varphi,g,f]\mapsto \langle\varphi\rangle$.
\end{constr}

We can now state our new duality theorem for $\HLC$:

\begin{theorem}\label{md} 
The functor $\overline{U}$ is an equivalence of categories. As a consequence, $\bf LKHaus$ is dually equivalent to the category
$\bf DBoo/\!\!\backsim$,
 whose objects are complete local contact algebras and whose morphisms are equivalence classes of\/ $\DBoo$-morphisms, to be composed by ordinary map composition of their representatives.
\end{theorem}

\doc
Just as $U$ induces the functor $\overline{U}$, the functor $W$ of Proposition \ref{Ul equivalencel} induces the functor
 $$\overline{W}:{\bf DBoo}/\!\backsim\,\longrightarrow {\sf C}({\cal A},{\cal P},{\cal X})/\!\sim,\quad \langle \varphi\rangle\mapsto [\varphi,g_\p,f_{\varphi}].$$
It maps objects as $W$ does and still satisfies $\overline{U}\circ\overline{W}={\rm Id}$ and $\overline{W}\circ\overline{U}\cong{\rm Id}$. Hence, $\overline{U}$ is an equivalence of categories, with quasi-inverse $\overline{W}$. The claimed duality now follows with Proposition \ref{Fed applicationl}.
\sqs


\section{A new proof of Dimov's extension of the de Vries duality theorem to the category $\HLC$}

In order to show that {\bf LKHaus} is dually equivalent to the category {\bf CLCA}
(see Definition \ref{dhc}),
according to Theorem \ref{md}, we just need to exhibit {\bf CLCA} as equivalent to the category ${\DBoo/\!\!\backsim}$ of Construction \ref{quotientfunctors1}.

\begin{pro}\label{V}
There is a functor
\begin{center}
$\xymatrix{{\bf DBoo}\ar[rr]^{V} & &{\bf CLCA},}$
\end{center}
which maps objects identically and sends a morphism $\varphi:(A,\smf,\BBBB)\lra (A\ap,\smf',\BBBB\ap)$ in {\bf DBoo} to
$\varphi\cuk$.
\end{pro}


\doc
For a morphism
$\varphi:A\to A\ap$
in {\bf DBoo}
we must first confirm that the map $\varphi\cuk$
(see Definition \ref{dhc} for this notation) satisfies conditions (CLC1-5) of
In what follows, we will omit the annotation with primes in $\smf$ and $\ll$.
 Condition (CLC1) holds trivially for $\varphi\cuk$. It is straightforward to show that, since $\varphi$ satisfies (CLC2), the map
 $\varphi\cuk$ satisfies (CLC2) and (CLC5).
 Let us  confirm that $\varphi\cuk$ satisfies (CLC3).
 Given $a\in A$, $b^*\in\BBBB$ and $b^*\ll a$,
 we have to show that $(\p\cuk(b))^*\ll\p\cuk(a)$.
  Since $\varphi$ preserves the Boolean negation,
  we have

\medskip

\begin{tabular}{lll}
$(\varphi\cuk(b))^*$&$=(\bigvee\{\varphi(e)\ |\ e\in \BBBB, e\ll b\})^*$ \\
&$=\bigwedge\;\{\varphi(e)^*\ |\ e\in \BBBB,  e\ll b\}$ \\
&$=\bigwedge\;\{\varphi(e^*)\ |\ e\in \BBBB,  e\ll b\}$. \\
 \end{tabular}

 \medskip

\noindent Analogously, we obtain
 $$(\varphi\cuk(a))^*=\bigwedge\;\{\varphi(c^*)\ |\ c\in \BBBB,  c\ll a\}.$$

\noindent  Assume we had $(\p\cuk(b))^*\not\ll\p\cuk(a)$. Then  $(\p\cuk(b))^*\smf(\p\cuk(a))^*$, i.e., $$(\bigwedge\;\{\varphi(e^*)\ |\ e\in \BBBB,  e\ll b\})\smf (\bigwedge\;\{\varphi(c^*)\ |\ c\in \BBBB,  c\ll a\}).$$ Now (BC2) implies that there exists $d\ap\in\BBBB\ap$ such that
 $$(d\ap\we\bigwedge\;\{\varphi(e^*)\ |\ e\in \BBBB,  e\ll b\})\smf (d\ap\we\bigwedge\;\{\varphi(c^*)\ |\ c\in \BBBB,  c\ll a\}),$$ Since $\p$ is a $\DBoo$-morphism, there exists $d\in\BBBB$ with $d\ap\le\p(d)$. Then  $(\p(d)\we\bigwedge\;\{\varphi(e^*)\ |\ e\in \BBBB,  e\ll b\})\smf (\p(d)\we\bigwedge\;\{\varphi(c^*)\ |\ c\in \BBBB,  c\ll a\})$ and thus $$(\bigwedge\;\{\varphi(d\we e^*)\ |\ e\in \BBBB,  e\ll b\})\smf (\bigwedge\;\{\varphi(d\we c^*)\ |\ c\in \BBBB,  c\ll a\}).$$
 There exists $c_0\in\BBBB$ such that $b^*\ll c_0\ll a$. Then $c_0^*\ll b$. Thus $d\we c_0^*\ll b$ and $d\we c_0^*\in\BBBB$. Using (BC1), we find an $e_0\in\BBBB$ such that $d\we c_0^*\ll e_0\ll b$, and $d\we c_0^*\ll e_0\vee d^*$ follows. Since $\p$ preserves the relation $\ll$, we obtain $\p(d\we c_0^*)\ll \p(e_0\vee d^*)$, i.e., $\p(d\we c_0^*)\nsmf \p(d\we e_0^*)$, a contradiction. Therefore, $(\p\cuk(b))^*\ll\p\cuk(a)$. So, condition (CLC3) is satisfied.

 For confirming that $\p\cuk$ satisfies (CLC4), let $b\ap\in\BBBB\ap$. We have to find a $b\in \BBBB$ with $\p\cuk(b)\ge b\ap$. Since $\p$ is a $\DBoo$-morphism, there exists a $b_0\in\BBBB$ such that $\p(b_0)\ge b\ap$. Using the fact that $b_0\ll 1$ and (BC1), we obtain that there exists a $b\in\BBBB$ such that $b_0\ll b$. Then $\p\cuk(b)=\bigvee\{\p(c)\st c\in\BBBB, c\ll b\}\ge\p(b_0)\ge b\ap$. Hence, condition (CLC4) is satisfied. This concludes the proof that $\p\cuk$ is a $\CLCA$-morphism.

For showing that $V$ preserves the identity map ${\rm id}_A$ on $(A,\smf,\BBBB)$, let $a\in A$. With $c\df\p\cuk(a)$ we have $c\le a$. Supposing that $c\neq a$, we obtain $a\we c^*\neq 0$ and thus, by (BC3), there exists $0\neq b\in\BBBB$ such that $b\ll a\we c^*$. Then $b\ll a$ and $b\le c^*$; thus $b\nleq c$, a contradiction. Therefore, $c=a$ showing that $V$ preserves the identity map on $A$.

Finally, we must show that $V$ preserves the composition. For morphisms $\varphi:(A,\smf,\BBBB)\lra (A\ap,\smf',\BBBB\ap)$ and $\psi:(A\ap,\smf',\BBBB\ap)\lra (A'',\smf'',\BBBB'')$ in {\bf DBoo}, we have to show that $V(\psi\circ\p)=V(\psi)\di V(\p)$, i.e. that $(\psi\circ\p)\cuk=\psi\cuk\di \p\cuk$, where $\psi\cuk\di \p\cuk=(\psi\cuk\circ \p\cuk)\cuk$. Indeed, since $\p$ and $\psi\cuk$ are monotone functions,
with Lemma \ref{assocuk} we obtain that $(\psi\circ\p)\cuk=(\psi\cuk\circ\p)\cuk=(\psi\cuk\circ\p\cuk)\cuk$. Hence, $V$ is a well-defined functor.
 \sqs

Unlike $U$, the functor $V$ of Proposition \ref{V} is {\em not}\/ an equivalence of categories. But it induces a functor $\ovl{V}$ which, as we will show, is even an isomorphism of categories.

\begin{constr}\label{quotientfunctors2}
\rm
 The functor $V$ induces a faithful functor
$$\xymatrix{{\bf DBoo}/\!\backsim\ar[rrr]^{\quad\overline{V}} & & &{\bf CLCA}},$$
mapping objects like  $V$, and morphisms by  $\langle\p\rangle\mapsto \p\cuk$.
\end{constr}

Since $\overline{V}$ maps objects identically and is trivially faithful, it suffices to show that $\overline{V}$ is full, in order for us to conclude that it is an isomorphism of categories.
We start by proving a couple of simple lemmata.

\begin{lm}\label{fl1}
Let $A$ be a local contact algebra and $\ult$ be a bounded ultrafilter in $A$. Then $\{a\in A\st\fa b\in\ult: a\smf b\}=\{a\in A\st\fa b\in\ult: a\smf_{\rm\! Al} b\}$ (with $\smf_{\rm\! Al}$ of {\rm Definition \ref{Alexprn}}), i.e., $\clu_\ult=\{a\in A\st\fa b\in\ult: a\smf b)\}$.
\end{lm}

\doc Let $a\in A$ and $a\smf_{\rm\! Al} b$ for all $b\in\ult$. We have to show that $a\smf b$ for all $b\in\ult$. Indeed, pick $b_0\in\ult\cap\BBBB$; then $b\we b_0\in \ult\cap\BBBB$  and, hence, $a\smf_{\rm\! Al} (b\we b_0)$ for every $b\in\ult$. This implies $a\smf (b\we b_0)$ and, thus, $a\smf b$ for every $b\in\ult$. The converse inclusion is obvious.
\sqs

\begin{lm}\label{fl2}
Let $A$ be a local contact algebra and $\ult,\ultv$ be  bounded ultrafilters in $A$.
Then, with the extension of the contact relation to ultrafilters (see Facts \ref{deVriesl}),  one has ($\ult\smf\ultv\iff\clu_\ult=\clu_\ultv$).
\end{lm}

\doc When $\ult\smf\ultv$, we show that $\clu_\ult\sbe\clu_\ultv$. Supposing that there exists $a\in\clu_\ult\stm\clu_\ultv$
 and using Lemma \ref{fl1},
 we find $d\in\ultv$ with $a\nsmf d$. Then, picking $e\in\ultv\cap\BBBB$, we have $e\we d\in\ultv$ and $a\nsmf (e\we d)$. Hence, $e\we d\ll a^*$, so that there exists $e_1\in\BBBB$ such that $e\we d\ll e_1\ll a^*$, i.e., $(e\we d)\nsmf e_1^*$ and $e_1\nsmf a$. Then $e_1\nin\ult$ because otherwise, since $a\in\clu_\ult$, we would have $e_1\smf a$. Now $e_1^*\in\ult$ follows, and since $e\we d\in\ultv$ and $\ult\smf\ultv$, we obtain  $(e\we d)\smf e_1^*$, a contradiction. Hence, $\clu_\ult\sbe\clu_\ultv$, with the reverse implication following by symmetry.

Conversely, if $\clu_\ult=\clu_\ultv$, then $\ult\smf\ultv$ follows from Lemma \ref{fl1} and the trivial inclusions $\ult\sbe\clu_\ult$ and $\ultv\sbe\clu_\ultv$.
\sqs

In the following proposition, the assertions (2) and (3) appeared in
 \cite{D-AMH1-10},
but the proofs given here are new. We use $\bar{\ep}_A^Z$ as defined in Notation \ref{zalgnl}; $\tau_{A}$ and $\bclust(A)$ were defined in the (sketch of the) proof of Theorem \ref{roeperl}, and $p_{A}:{\rm BUlt}(A)\to{\rm BUlt}(A)/\!\smallfrown$ is the irreducible perfect natural quotient map of Proposition \ref{eqrel}(d).

\begin{pro}\label{full propl} {\rm(1)} For every complete local contact algebra $A$, with $Z_A\df \BUlt(A)$, one has the homeomorphism
$$\gamma_{A}:Z_A/\!\smallfrown\;\longrightarrow {\sf BClust}(A),\;[\mathfrak u]\longmapsto\mathfrak c_{\mathfrak u}=\{a\in A \ |\ \forall b\in\mathfrak u:a\smallfrown b\},$$
where $Z_A$ is considered as a  subspace of\/ $\Ultsf(A)$,
as well as the Boolean isomorphism
$$\rho_{\gamma_{A}\circ p_{A}}:{\sf CO}(Z_A)\longrightarrow{\sf RC}( \bclust(A)),\;H\longmapsto\gamma_{A}(p_{A}(H)),$$
which fits into the commutative diagram
\begin{center}
$\xymatrix{&  &A\ar[lld]_{\bar{\ep}_A^{Z_A}}\ar[drr]^{\tau_{A}} &  &\\
{\sf CO}(Z_A)\ar[rrrr]^{\rho_{\gamma_{A}\circ p_{A}}} & & & & \rc(\bclust(A))\,.}$
\end{center}
\noindent In particular, $\bclust(A)$ is a locally compact Hausdorff space.

\bigskip

{\rm (2)} $\tau_{A}: A\lra \rcl(\bclust(A))$ is a  {\bf CLCA}-isomorphism.

\bigskip

{\rm (3)}
For every morphism $\alpha:(A,\smf,\BBBB)\lra (A',\smf',\BBBB\ap)$ in\/ $\bf CLCA$, one has the continuous map
$$\hat{\alpha}:\bclust(A\ap)\lra\bclust(A),\quad {\mathfrak c'}\mapsto \uparrow\!\{b\in \BBBB\ |\ \forall\, a\in \!A\; (\,b\ll a \Longrightarrow \alpha(a)\in\mathfrak c'\,)\},$$
which fits into the commutative diagram
\begin{center}
$\xymatrix{
A\ar[rr]^{\alpha}\ar[d]_{\tau_{A}} & & A\ap\ar[d]^{\tau_{A\ap}}\\
{\sf RCL}(\bclust(A))\ar[rr]^{{\sf RCL}(\hat{\alpha})} & & {\sf RCL}( \bclust(A\ap))\,.\\
}$	
\end{center}
\end{pro}

\doc   We put $L_A\df \bclust(A)$.

\smallskip

(1) We first note that the family $\{\tau_A(a)\ |\ a\in A\}$ can indeed be taken as a closed base for the space
$L_A$
since $\tau_A(0)=\emptyset$, and for all $a,b\in A$ one has
$$\tau_A(a\vee b)=\{\mathfrak c\in L_A \st a\vee b\in\mathfrak c\}=\{\mathfrak c\in L_A\ |\ a\in\mathfrak c\mbox{ or }b\in\mathfrak c\}=\tau_A(a)\cup\tau_A(b).$$
Further, by Lemma \ref{fl2}, the map $\g_A$ is a well-defined injection. It is also a surjection because, when $\clu\in L_A$, then there exist $b\in\BBBB\cap\clu$ and $\ult\in\Ult(A)$ such that $b\in\ult\sbe\clu_\ult=\clu$ (see Facts \ref{deVriesl} and Lemma \ref{fl1}), so that $\clu=\g_A([\ult])$. Therefore,
 $\gamma_A$ is a well-defined bijective map.
 Also, setting
 $$Y_A \df Z_A/\!\smallfrown$$
  and applying Proposition \ref{eqrel}(d) and Alexandroff's Theorem (see Facts \ref{projectivel}) to $p_A:Z_A\to Y_A$, we obtain the Boolean isomorphism
$$\rho_{p_A}:{\sf CO}(Z_A)={\sf RC}(Z_A)\to {\sf RC}(Y_A),\quad H\mapsto p_A(H).$$
For showing that $\gamma_A$ is a homeomorphism, we now prove that $\gamma_A(p_A(\epb(a)))=\tau_A(a)$ for every $a\in A$. Indeed, for $a\in A$ one has
$\gamma_A(p_A(\epb(a)))=\g_A(\{[\ult]\st a\in\ult\in Z_A\})=\{{\mathfrak c}_{\mathfrak u}\ |\ a\in\mathfrak u\in Z_A\}\sbe\{\mathfrak c \in L_A\ |\ a\in\mathfrak c\}=\tau_A(a)$. For proving the converse inclusion, let $\clu\in \tau_A(a)$. Then, by Facts \ref{deVriesl}, there exists $\ult\in\Ult(A)$ such that $a\in\ult\sbe\clu_\ult=\clu$. If $\ult\nin Z_A$, then $\ult\sbe A\stm\BBBB$. Since $A\stm\BBBB$ is a cluster in the normal contact algebra $(A,\smf_{\rm\! Al})$ (see Lemma \ref{neogrn}) and $\clu_\ult$ is the unique cluster in the normal contact algebra $(A,\smf_{\rm\! Al})$ containing $\ult$ (by Facts \ref{deVriesl}), we conclude that $A\stm \BBBB=\clu_\ult=\clu\in Z_A$, a contradiction. Hence, $\ult\in Z_A$. Therefore, $\gamma_A(p_A(\epb(a)))=\tau_A(a)$ for every $a\in A$. Since
$\epb(A)=Z_A\cap\ep_A(A)=Z_A\cap\CO(\Ultsf(A))=Z_A\cap\RC(\Ultsf(A))=\RC(Z_A)$, we obtain that
$\RC(Y_A)=\{p_A(\epb(a))\st a\in A\}$. Now, the facts that $\RC(Y_A)$ is a closed base for $Y_A$ and  the family $\{\tau_A(a)\ |\ a\in A\}$ is  a closed base for the space $L_A$ imply that $\g_A$ is a homeomorphism.
Consequently,
$$\rho_{\gamma_A}:{\sf RC}(Y_A)\to{\sf RC}(L_A),\quad K\mapsto \gamma_A(K),$$
and then also $\rho_{\gamma_A\circ p_A}=\rho_{\gamma_A}\circ\rho_{p_A}$, are Boolean isomorphisms, and the triangle of (1) commutes.
Finally, since $Y_A$ is a locally compact Hausdorff space (as shown in Proposition \ref{eqrel}(d)) and $\gamma_A$ is a homeomorphism, $L_A$ is also a locally compact Hausdorff space.

\smallskip

(2) Clearly, by (1), $\tau_A:A\to\rc(L_A)$ is a Booolean isomorphism. With the assertions of Facts \ref{deVriesl}, Lemma \ref{fl2} and Proposition \ref{bbclf}, for all $a,b\in A$ one has

\medskip

\begin{tabular}{ll}
$ a\smallfrown b$ & $\iff \ex d\in\BBBB : (a\we d)\smf (b\we d)$\\
 &$\iff \exists\, \mathfrak u,\mathfrak v\in Z_A: a\in\mathfrak u,\, b\in\mathfrak v,\,\mathfrak u\smallfrown \mathfrak v$\\
 &$\iff \exists\, \mathfrak u,\mathfrak v\in Z_A: a\in\mathfrak u,\,b\in\mathfrak v,\,\mathfrak c_{\mathfrak u}=\mathfrak c_{\mathfrak v} $ \\
 &$\iff \exists\,\mathfrak c\in L_A: a,b\in\mathfrak c$\\
 & $\iff \tau_A(a)\cap\tau_A(b)\nes$\\
 & $\iff \tau_A(a)\smallfrown_{L_A}\tau_A(b).$\\
  \end{tabular}

  \medskip

 \noindent Consequently, $\tau_A$ preserves the relation $\ll$. We now show $\tau_A(\BBBB)=\CR(L_A)$. Indeed, by Proposition \ref{eqrel}(b) we know ($b\in\BBBB\iff \ep_A(b)\sbe Z_A\iff \ep_A(b)\in\KO(Z_A)$). In the proof of Proposition \ref{Ul equivalencel} we showed $(A,Z_A,p_A)\in|\CAPX|$. Thus, by Fact \ref{creql}, $p_A(\KO(Z_A))=\CR(Y_A)$. Since $\ep_A(\BBBB)=\KO(Z_A)$, we obtain $p_A(\ep_A(\BBBB))=\CR(Y_A)$. By (1), $\g_A$ is a homeomorphism, so that $\tau_A(\BBBB)=\g_A(p_A(\ep_A(\BBBB)))=\CR(L_A)$.
 In summary $\tau_A:A\lra\rcl(L_A)$ is a {\bf CLCA}-isomorphism.

 \medskip

(3) We must first confirm that $\clu\df\hat{\alpha}(\mathfrak c')$ is a bounded cluster of $A$, for every $\mathfrak c'\in L_{A\ap}$. 
Putting $\cbf\df \{b\in \BBBB\ |\ \forall\, a\in \!A\; (\,b\ll a \Longrightarrow \alpha(a)\in\mathfrak c'\,)\}$ we have $\clu=\uparrow\!\cbf$. Let us show that $\cbf\nes$. Indeed, there exists $d\in\clu\ap\cap\BBBB\ap$. Hence, there exists $b\in\BBBB$ such that $d\le\a(b)$. Then, for every $a\in A$ such that $b\ll a$, we have $d\le\a(b)\le\a(a)$ and thus $\a(a)\in\clu\ap$. Therefore, $b\in\cbf$ and thus $\clu\nes$,
so that (cl\,1) is satisfied.

For (cl\,2), consider $a_1,a_2\in\clu$. Then there exist $b_1,b_2\in\cbf$ such that $b_i\le a_i$, $i=1,2$.
Assume that we had $b_1\not\smallfrown b_2$, that is: $b_1\ll b_2^*$. Choosing $b_3,b_4\in \BBBB$ such that $b_1\ll b_3\ll b_4\ll b_2^*$, we obtain that $\a(b_3)\ll\a(b_4)$, i.e., $\a(b_3)\nsmf\ap(\a(b_4))^*$. By Lemma \ref{pf1}(a), we have that $\a(b_4^*)\le(\a(b_4))^*$, so that $\a(b_3)\nsmf\ap\a(b_4^*)$. The definition of $\cbf$ implies that $\a(b_3), \a(b_4^*)\in\clu\ap$ and thus $\a(b_3)\smf'\, \a(b_4^*)$, a contradiction. Hence $b_1\smf b_2$ and thus $a_1\smf a_2$.

To confirm (cl\,3), we consider $a_1,a_2\in A$ with $a_1\vee a_2\in\clu$. There exists $b\in\cbf$ such that $b\le a_1\vee a_2$. Put $b_i\df b\we a_i$, $i=1,2$. Then $b_i\le a_i$ and $b_i\in\BBBB$, $i=1,2$. Since $b=b_1\vee b_2$, we obtain $b_1\vee b_2\in\cbf$.
Assume we had
 $b_1\not\in\cbf$ and $b_2\not\in\cbf$. Then there are $c_1,c_2\in A$ such that $b_i\ll c_i$ and $\alpha(c_i)\not\in\mathfrak c'$, $i=1,2$.
 There exist $b_{11},b_{22}\in\BBBB$ such that $b_i\ll b_{ii}\ll c_i$, $i=1,2$. Then, by Lemma \ref{pf1}(b), $\a(b_{11}\vee b_{22})\ll \a(c_1)\vee\a(c_2)$. Since $b\in\cbf$ and $b\ll b_{11}\vee b_{22}$, we obtain $\a(b_{11}\vee b_{22})\in\clu\ap$ and, thus, $\a(c_1)\vee\a(c_2)\in\clu\ap$. Hence, $\a(c_1)\in\clu\ap$ or $\a(c_2)\in\clu\ap$, a contradiction. Therefore, $b_1\in\cbf$ or $b_2\in\cbf$. This implies that $a_1\in\clu$ or $a_2\in\clu$.

Since, clearly, $\clu$ is upward closed, we obtain that it is a clan. For proving that $\clu$ is a cluster,
 note first that $\cbf=\clu\cap\BBBB$. Indeed, the inclusion $\cbf\sbe\clu$ is clear; conversely, if $a\in\clu\cap\BBBB$, then there exists $b\in\cbf$ with $b\le a$, which implies, with the definition of $\cbf$, that $a\in\cbf$. We  now show $\clu=\clu_\ult$ for some $\ult\in Z_A$.

 We pick $b_0\in\cbf=\clu\cap\BBBB$. By Facts \ref{deVriesl},  there exists an $\ult\in Z_A$ such that $b_0\in\ult\sbe\clu$. Then, by Lemma \ref{fl1} and Facts \ref{deVriesl}, $\clu_\ult\df\{a\in A\st \fa b\in\ult:a\smf b\}$ is a cluster in the normal contact algebra $(A,\smf_{\rm\! Al})$. Since $\clu$ is a clan (and, thus, satisfies (cl\,2)), we obtain $\clu\sbe \clu_\ult$. For proving the converse inclusion, consider $b\in\clu_\ult\cap\BBBB$. Then $b\smf \clu$ (which means $b\smf a$ for all $a\in\clu$), because $\clu\sbe\clu_\ult$. Assume we had $b\nin\cbf$. Then there exists $a\in A$ such that $b\ll a$ and $\a(a)\nin\clu\ap$. There exists $b_1\in\BBBB$ with $b\ll b_1\ll a$. Then $(\a(b_1^*))^*\ll\a(a)$, i.e., $(\a(a))^*\ll \a(b_1^*)$. Since $\a(a)\nin\clu\ap$, we obtain $(\a(a))^*\in\clu\ap$ and, hence, $\a(b_1^*)\in\clu\ap$.
 We now show that $b_1^*\in\clu$. Indeed, by Proposition \ref{bbcl1}(a), there exists $b\ap\in\clu\ap\cap\BBBB\ap$ such that $b\ap\le\a(b_1^*)$. Then there exists $b_2\in\BBBB$ with $b\ap\le\a(b_2)$ and, hence, $b\ap\le\a(b_2)\we\a(b_1^*)=\a(b_2\we b_1^*)$. Then $b_2\we b_1^*\in\BBBB$ and $b_2\we b_1^*\in\cbf$ follows. Since $b_2\we b_1^*\le b_1^*$, we obtain $b_1^*\in\clu\sbe\clu_\ult$ and, hence, $b\smf b_1^*$. Since $b\ll b_1$, we have a contradiction. Therefore, $b\in\cbf$. So, we have shown $\clu_\ult\cap\BBBB\sbe\cbf=\clu\cap\BBBB\sbe\clu_\ult\cap\BBBB$, i.e., $\cbf=\clu_\ult\cap\BBBB$. Then, by  Proposition \ref{bbcl1}(a), for every $a\in\clu_\ult$ there exists $b\in\cbf$ with $b\le a$, which implies $a\in\clu$. Therefore, $\clu_\ult\sbe\clu\sbe\clu_\ult$, i.e., $\clu=\clu_\ult$, and we conclude $\clu$ is a bounded cluster in $A$.

To show the continuity of $\hat{\alpha}$, we first observe that since $\{\tau_A(a)\ |\ a\in A\}={\rm RC}(L_A)$ is a base for closed sets in $L_A$ that coincides with  the set
 $\{(\tau_A(a))^*\ |\ a\in A\}$, the family
 $${\rm RO}(L_A)\df\{{\rm int}(F)\ |\ F\in{\sf RC}(L_A)\}=\{{\rm int}(\tau_A(a))\ |\ a\in A\}$$
is a base of open sets in $L_A$. We must therefore show that $\hat{\alpha}^{-1}({\rm int}(\tau_A(a)))$ is open in $L_{A\ap}$, for every $a\in A$.
To this end we prove for $a\in A$ that
$$\hat{\alpha}^{-1}({\rm int}(\tau_A(a)))=\bigcup\{\int(\tau_{A\ap}(\a(b)))\st b\in\BBBB, b\ll a\}.$$
Indeed, denoting the Boolean operations of ${\sf RC}(L_A)$ as in Facts \ref{deVriesAlgebraicl}, with $F= \tau_A(a)$ we have
$$\int(\tau_A(a))=\mbox{{\rm int}}(F)=L_A\setminus\mbox{{\rm cl}}(L_A\setminus F)=L_A\setminus F^*=L_A\setminus \tau_A(a^*)=\{\mathfrak c\in L_A\ |\ a^*\not\in\mathfrak c\}.$$
Hence, for every $\mathfrak c'\in L_{A\ap}$,  using Proposition \ref{bbcl1}(b), one obtains
$$\mathfrak c'\in\hat{\alpha}^{-1}({\rm int}(F))\iff \hat{\a}(\clu\ap)\in\int(F)\iff a^*\not\in \hat{\alpha}(\mathfrak c')\iff(\exists\, b\in \hat{\a}(\clu\ap)\cap\BBBB) (b\ll a).$$
Then there exist $b_1,b_2\in\BBBB$ such that $b\ll b_1\ll b_2\ll a$. Now, the definition of $\hat{\a}(\clu\ap)\cap\BBBB$ implies $\a(b_1)\in\clu\ap$ and, thus, $\clu\ap\in\tau_{A\ap}(\a(b_1))$. Since $\a(b_1)\ll\a(b_2)$, we obtain  $\tau_{A\ap}(\a(b_1))\sbe \int(\tau_{A\ap}(\a(b_2)))$, and
$$\hat{\alpha}^{-1}({\rm int}(\tau_A(a)))\sbe\bigcup\{\int(\tau_{A\ap}(\a(b)))\st b\in\BBBB, b\ll a\}.$$
follows. For proving the converse inclusion, let $b\in\BBBB$, $b\ll a$ and $\clu\ap\in\int(\tau_{A\ap}(\a(b)))$. Then $(\a(b))^*\nin\clu\ap$ and thus $\a(b)\in\clu\ap$. This implies that for every $c\in A$ with $b\ll c$, we have $\a(c)\in\clu\ap$ because $\a(c)\ge\a(b)$. Hence, $b\in\hat{\a}(\clu\ap)\cap\BBBB\sbe \hat{\a}(\clu\ap)$ and, thus, $\hat{\a}(\clu\ap)\in\tau_A(b)$ and, since $b\ll a$, $\tau_A(b)\sbe \int(\tau_A(a))=\int(F)$. Therefore, $\clu\ap\in \hat{\alpha}^{-1}({\rm int}(\tau_A(a)))$.
We conclude that the set $\hat{\alpha}^{-1}({\rm int}(F))$
is indeed open in $L_{A\ap}$.
This completes the proof that $\hat{\alpha}$ is a continuous function.

Furthermore, continuing to take advantage of the Boolean isomorphism $\tau_A$ and the fact that  $\alpha$ is a $\CLCA$-morphism, we see that

\medskip

\begin{tabular}{lll}
$\mbox{{\rm cl}}(\hat{\alpha}^{-1}({\rm int}(\tau_A(a))))$&$=\mbox{{\rm cl}}(\bigcup\{\int(\tau_{A\ap}(\a(b)))\st b\in\BBBB, b\ll a\})$\\
&$=\bigvee\{\tau_{A\ap}(\a(b))\st b\in\BBBB, b\ll a \}$\\
&$=\tau_{A^{\,\prime}}(\bigvee\{\alpha(b)\st b\in\BBBB, b\ll a\})$\\
&$=\tau_{A^{\,\prime}}(\alpha(a))$\\
\end{tabular}

\medskip

\noindent holds in $L_{A\ap}$. But this proves precisely the claimed identity ${\sf RCL}(\hat{\alpha})\circ\tau_A=\tau_{A'}\circ\alpha$.
\sqs 	

As a final step in completing the proof of the fullness of the functor $\overline{V}$ we prove the following lemma.

\begin{lm}\label{mainlml}
Let $X$ and $Y$ be locally compact Hausdorff spaces, $p:X\to Y$ be a perfect irreducible map,  $f:X'\to X,\,g:Y'\to Y$ be continuous maps of topological spaces, and let $p\ap:X'\to Y'$ be a closed irreducible map with $p\circ f=g\circ p\ap$.  Then every $G\in\operatorname{RC}(X)$ satisfies the identity
$$\mbox{{\rm cl}}(g^{-1}(\mbox{{\rm int}}(p(G))))=\bigvee\{p^{\,\prime}({\rm cl}(f^{-1}({\rm int}(H))))\ |\ H\in\operatorname{CR}(X), p(H)\subseteq\mbox{{\rm int}}(p(G))\}.$$
\end{lm}

\doc
Alexandroff's Theorem (see Facts \ref{projectivel}) gives us the Boolean isomorphism $$\rho_p:\operatorname{RC}(X)\longrightarrow \operatorname{RC}(Y),\ \ H\mapsto p(H),$$
  whose inverse maps $K\in\operatorname{RC}(Y)$ to $\mbox{{\rm cl}}(p^{-1}(\mbox{{\rm int}}(K)))$. The proof of Fact \ref{creql} shows $\rho_p(\CR(X))=\CR(Y)$. Now,
  since $Y$ is a locally compact $T_3$-space, for every $G\in\operatorname{RC}(X)$ we
  obtain  $${{\rm int}}(p(G))=\bigcup\{p(H)\ |\ H\in\operatorname{CR}(X)\mbox{  and }p(H)\subseteq\mbox{{\rm int}}(p(G))\}.$$
  In describing the inverse image under $g$ of this union, we first note that, since $p\ap$ is surjective, we have
   $g^{-1} (p(H))=p\ap((p\ap)\inv (g^{-1}(p(H))))=p\ap(f^{-1}(p^{-1}(p(H)) ))$
  for every $H\in\operatorname{CR}(X)$, which gives us
   $$g^{-1}(\mbox{{\rm int}}(p(G)))=\bigcup\{p^{\,\prime}(f^{-1}(p^{-1}(p(H))))\ |\ H\in\operatorname{CR}(X)\mbox{ and }p(H)\subseteq\mbox{{\rm int}}(p(G))\}.$$
   Since, for $H\in\CR(X)$, $p(H)$ is compact and $Y$ is a locally compact $T_3$-space, there exists $H\ap\in\CR(X)$ such that
   $$p(H)\sbe\int(p(H\ap))\sbe p(H\ap)\sbe\int(p(G)).$$
   Obviously, $H\ap=\rho_p\inv(\rho_p (H\ap))=\cl(p\inv(\int(p(H\ap)))).$
  Consequently,
  $$H\subseteq p^{-1}(p(H))\subseteq p^{-1}(\mbox{{\rm int}}(p(H')))\subseteq\mbox{{\rm cl}}(p^{-1}(\mbox{{\rm int}}(p(H'))))= H'.$$
  As a result,
   $$g^{-1}(\mbox{{\rm int}}(p(G)))=\bigcup\{p^{\,\prime}(f^{-1}(H))\ |\ H\in\operatorname{CR}(X)\mbox{ and }p(H)\subseteq\mbox{{\rm int}}(p(G))\}.$$
   Also, since  $H\sbe p\inv(p(H))\sbe p\inv(\int(p(H\ap)))\sbe \int(\cl(p\inv(\int(p(H\ap)))))=\int(H\ap)$, for every $H\in\CR(X)$, we obtain that
   $f^{-1}(H)\subseteq f^{-1}(\mbox{{\rm int}}(H')) \subseteq \mbox{{\rm cl}}(f^{-1}(\mbox{{\rm int}}(H'))) \subseteq f^{-1}(H')$. Thus
    $$g^{-1}(\mbox{{\rm int}}(p(G)))=\bigcup\{p^{\,\prime}({\rm cl}(f^{-1}({\rm int}(H))))\ |\ H\in\operatorname{CR}(X)\mbox{ and }p(H)\subseteq\mbox{{\rm int}}(p(G))\}.$$
  Finally, taking closures of both sides of this formula and using  Facts \ref{deVriesAlgebraicl}, we conclude the claimed formula.
 \sqs

We can now sum up, return to the functor
$$\overline{V}:{\bf DBoo}/\!\!\backsim\;\longrightarrow{\bf CLCA}$$
of Construction \ref{quotientfunctors2} and complete our alternative proof of the extension of de Vries' dual equivalence to the category of locally compact Hausdorff spaces and continuous maps obtained in \cite{D-AMH1-10}.

\begin{theorem}\label{newlcdth}
The functor
$\overline{V}$
is an isomorphism of categories.
Consequently, the category $\bf LKHaus$ is dually equivalent to the category $\bf CLCA$.
\end{theorem}

\doc
In order to show that $\overline{V}$ is full, given $\alpha:(A,\smf,\BBBB)\to (A',\smf',\BBBB\ap)$ in $\bf CLCA$, we must show  $\alpha=V\varphi$, for some $\varphi:A\to A\ap$ in
$\bf DBoo$. Proposition \ref{full propl} produces the continuous map $\hat{\alpha}:{\sf BClust}(A')\to{\sf BClust}(A)$, as well as the homeomorphism
$\gamma_{A}:Z_A/\!\smallfrown\;\longrightarrow {\sf BClust}(A),\;[\mathfrak u]\mapsto\mathfrak c_{\mathfrak u}$, where $Z_A=\BUlt(A)$, and the maps $\tau_{A}:A\lra \rcl(\bclust(A))$ and $p_{A}: Z_A\to Z_A/\!\smallfrown$.
As in the proof of Proposition \ref{full propl}, we set $L_A=\bclust(A)$.

We put $\bar{f}=\gamma_A^{-1}\circ\hat{\alpha}\circ\gamma_{A'}$. Since $p_A$ is a perfect surjection, Rump's Proposition \ref{rumped} shows that there exists a continuous map $f:Z_{A\ap}\to Z_A$  making the diagram
\begin{center}
$\xymatrix{Z_A \ar[d]_{p_A} & Z_{A\ap} \ar[l]_{f}\ar[d]^{p_{A'}}\\
Z_A/\!\smallfrown\ar[d]_{\gamma_A} & Z_{A\ap}/\!\smallfrown\ap\ar[l]_{\bar{f}}\ar[d]^{\gamma_{A'}}\\
L_A & L_{A\ap}\ar[l]^{\hat{\alpha}}\\
}$	
\end{center}
commute. By Theorem \ref{zdualityed1}, there exists a unique Boolean homomorphism $\varphi:A\to A'$ such that $(\p,f):(A,Z_A)\lra (A\ap,Z_{A\ap})$ is a $\LZCB$-morphism, i.e., $\Ultsf(\p)(Z_{A\ap})\sbe Z_A$, and the restriction of $\Ultsf(\p)$ on $Z_{A\ap}$ coincides with $f$.
Obviously,
the commutativity of the above diagram implies that $f$ preserves the contact relation for ultrafilters (i.e., if $\ult\ap,\ultv\ap\in Z_{A\ap}$ and $\ult\ap\smf'\,\ultv\ap$ then $f(\ult\ap)\smf f(\ultv\ap)$). We will now show that
 $\varphi$  reflects the contact relation. Indeed, let $a,b\in A$ and $\p(a)\smf'\, \p(b)$. Then, by (BC2), there exist $a\ap,b\ap\in\BBBB\ap$ such that $a\ap\le\p(a)$, $b\ap\le\p(b)$ and $a\ap\smf'\, b\ap$. Now, by Facts \ref{deVriesl}, there exist $\ult\ap,\ultv\ap\in\Ult(A\ap)$ with $a\ap\in\ult\ap$, $b\ap\in\ultv\ap$ and $\ult\ap\smf'\,\ultv\ap$. Since $a\ap,b\ap\in\BBBB\ap$, we obtain that $\ult\ap,\ultv\ap\in Z_{A\ap}$. Also, $\p(a)\in\ult\ap$ and $\p(b)\in\ultv\ap$. Since $f(\ult\ap)\smf f(\ultv\ap)$, i.e., $\p\inv(\ult\ap)\smf\p\inv(\ultv\ap)$, we conclude $a\smf b$. So, $\varphi$  reflects the contact relation.

 We now show that for every $b\ap\in\BBBB\ap$ there exists $b\in\BBBB$ such that $b\ap\le\p(b)$. Indeed, for $b\ap\in\BBBB\ap$ one has $\ep_{A\ap}(b\ap)\sbe Z_{A\ap}$, and $\ep_{A\ap}(b\ap)$ is compact. Thus, $K=f(\ep_{A\ap}(b\ap))=\p\inv(\ep_{A\ap}(b\ap))$ is a compact subspace of $Z_A$. Then, since $Z_A$ is an open subset of $\Ultsf(A)$, there exists $b\in\BBBB$ such that $K\sbe \ep_A(b)$. So, $K=\Ultsf(\p)(\ep_{A\ap}(b\ap))\sbe \ep_A(b)$. Then, using Stone duality, we obtain that $\ep_{A\ap}(b\ap)\sbe(\Ultsf(\p))\inv(\ep_A(b))=(\COsf(\Ultsf(\p)))(\ep_A(b))
 =\ep_{A\ap}(\p(b))$. Therefore, $b\ap\le\p(b)$.
 This concludes the proof that $\p$ is
 a  $\bf DBoo$-morphism.

Applying Lemma \ref{mainlml} to the outer rectangle of the above diagram, thus
putting $p\df\gamma_A\circ p_A,\, p^{\,\prime}\df \gamma_{A^{\,\prime}}\circ p_{A^{\,\prime}}$ and $g\df\hat{\alpha}$,
 we obtain that for every $G\in\RC(Z_A)$,
 $$\mbox{{\rm cl}}(g^{-1}(\mbox{{\rm int}}(p(G))))=\bigvee\{p^{\,\prime}({\rm cl}(f^{-1}({\rm int}(H))))\ |\ H\in\operatorname{CR}(Z_A), p(H)\subseteq\mbox{{\rm int}}(p(G))\}.$$
 The above formula,  together with the identities $\RC(Z_A)=\CO(Z_A)$ and  $\CR(Z_A)=\KO(Z_A)$, implies that for every $G\in\CO(Z_A)$,
 $$\mbox{{\rm cl}}(g^{-1}(\mbox{{\rm int}}(p(G))))=\bigvee\{p^{\,\prime}(f\inv(H))\ |\ H\in\operatorname{KO}(Z_A), p(H)\ll_{L_A} p(G)\}.$$
 Since $\rho_p:\CO(Z_A)\to\RC(L_A),\ \  F\mapsto p(F),$ is a Boolean isomorphism,
 and since $\rho_p(\CR(Z_A))=\CR(L_A)$ (see the proof of Fact \ref{creql}), i.e.,
 $\rho_p(\KO(Z_A))=\CR(L_A)$,
 the above equality can be rewritten in the following form: for every $F\in\RC(L_A)$,
 $$\cl(g\inv(\int(F)))=\bigvee\{\rho_{p\ap}(f\inv(\rho_p\inv(Q)))\st Q\in\CR(L_A), Q\ll_{L_A} F\}.$$
 Hence,
 $$\mbox{{\sf RCL}}(\hat{\alpha})=(\rho_{p^{\,\prime}}\circ \mbox{{\sf CO}}(\mbox{{\sf Ult}}(\varphi))\circ\rho_p^{-1})\cuk.$$
 Now, Proposition \ref{full propl}(3) implies
 $\tau_{A^{\,\prime}}\circ\alpha\circ\tau_A^{-1}=(\rho_{p^{\,\prime}}\circ \mbox{{\sf CO}}(\mbox{{\sf Ult}}(\varphi))\circ\rho_p^{-1})\cuk$ and, thus,
 $$\alpha=\tau_{A^{\,\prime}}^{-1}\circ\zeta\cuk\circ\tau_A,\ \ \mbox{where}\ \  \zeta:=\rho_{p^{\,\prime}}\circ \mbox{{\sf CO}}(\mbox{{\sf Ult}}(\varphi))\circ\rho_p^{-1}.$$
 With  Lemma \ref{assocuk} and Proposition \ref{full propl}(2) we then obtain

 \medskip

 \begin{tabular}{ll}
$\alpha$ & $=\alpha\cuk =((\tau_{A^{\,\prime}}^{-1}\circ\zeta\cuk)\circ\tau_A)\cuk=
((\tau_{A^{\,\prime}}^{-1}\circ\zeta\cuk)\cuk\circ\tau_A\cuk)\cuk$\\ & $=
(((\tau_{A^{\,\prime}}^{-1})\cuk\circ\zeta\cuk)\cuk\circ(\tau_A)\cuk)\cuk$\\ & $=
 ((\tau_{A^{\,\prime}}^{-1}\circ\zeta)\cuk\circ(\tau_A)\cuk)\cuk =
 (\tau_{A^{\,\prime}}^{-1}\circ\zeta\circ\tau_A)\cuk$\\ & $=(\tau_{A^{\,\prime}}^{-1}\circ\rho_{p^{\,\prime}}\circ \mbox{{\sf CO}}(\mbox{{\sf Ult}}(\varphi))\circ\rho_p^{-1}\circ\tau_A)\cuk.$\\
 \end{tabular}

 \medskip
\noindent From Proposition \ref{full propl}(1) we obtain ${\bar{\ep}}_A^{Z_A}=\rho_p\inv\circ\tau_A$ and $({\bar{\ep}}_{A\ap}^{Z_{A\ap}})\inv=\tau_{A\ap}\inv\circ\rho_{p\ap}$.
Hence, using the equality ${\bar{\ep}}_{A\ap}^{Z_{A\ap}}\circ\p=(\COsf(\Ultsf(\p))|Z_{A\ap})\circ{\bar{\ep}}_A^{Z_A}$ (see \cite[arXiv version, p. 12, line 2]{DD} or the proof of Theorem \ref{zdualityedll}), we conclude that
  $$\alpha=(({\bar{\varepsilon}}_{A^{\,\prime}}^{Z_{A\ap}})^{-1}\circ \mbox{{\sf CO}}(\mbox{{\sf Ult}}(\varphi))\circ {\bar{\varepsilon}}_A^{Z_A})\cuk=\varphi\cuk=V(\varphi),$$
as desired.
\sqs

We also confirm that our constructions leading up to Theorem \ref{newlcdth} gives, up to natural isomorphisms, the functors $\sf BClust$ and $\sf RCL$ furnishing the dual equivalence with unit $\sigma$ and counit $\tau$ as described in the proof of Theorem \ref{lccont}; that is:

\begin{cor}\label{cornlcdth}
The diagram
\begin{center}
$\xymatrix{{\bf CLCA}^{\rm op}\ar@/^0.2pc/[r]^{\overline{V}^{-1}\;\,}\ar@/^2.4pc/[rrr]^{\sf BClust} & ({\bf DBoo}/\!\backsim)^{\rm op}\ar@/^0.2pc/[r]^{\overline{W}\quad}\ar@/^0.2pc/[l]^{\overline{V}\;\;} & ({\sf C}({\cal A},{\cal P},{\cal X})/\!\sim)^{\rm op}\ar@/^0.2pc/[r]^{\quad\tilde{T}}\ar@/^0.2pc/[l]^{\overline{U}\quad} & {\bf LKHaus}\ar@/^0.2pc/[l]^{\quad\tilde{S}}\ar@/^2.4pc/[lll]^{\sf RCL}\\
}$	
\end{center}
commutes in an obvious sense, up to natural isomorphism.	
\end{cor}

\doc
It suffices to confirm that the functor $\tilde{T}\circ\overline{W}\circ\overline{V}^{-1}$ is naturally isomorphic to $\sf BClust$ since then its left adjoint, $\overline{V}\circ\overline{U}\circ\tilde{S}$, must be naturally isomorphic to $\sf RCL$. Indeed, for $(A,\smallfrown,\BBBB)\in |{\bf CLCA}|$, by Proposition \ref{full propl} one has the homeomorphism
$$\xymatrix{\tilde{T}(\overline{W}(\overline{V}^{-1}(A,\smallfrown,\BBBB)))
=\tilde{T}(A,\BUlt(A),p_A)={\rm BUlt}(A)/\!\smallfrown\ar[rr]^{\qquad\qquad\qquad\qquad\qquad\gamma_A} & & {\sf BClust}(A).\\
}$$
Checking that $\gamma_A$ is natural in $A$ involves going back to the morphism definitions of the functors involved, but that is a routine matter.	
\sqs

\end{document}